\documentclass[11pt]{article}
\usepackage{graphicx, color, epsfig, verbatim, subfigure, epsf} 
\usepackage{amsmath}
\usepackage{amssymb}
\usepackage{amsfonts}
\usepackage{amsthm}
\usepackage{latexsym}
\usepackage{fancyhdr,lastpage}
\usepackage{xspace}
\usepackage{setspace}
\usepackage[affil-it]{authblk}

\newtheorem{theorem}{Theorem}
\newtheorem{lemma}[theorem]{Lemma}
\newtheorem{corollary}[theorem]{Corollary}
\newtheorem{Pro}{Proposition}
\newtheorem*{Mth}{Theorem}
\theoremstyle{remark}

\newtheorem*{remark}{Remark}

\numberwithin{theorem}{section}
\numberwithin{Pro}{section}

\usepackage{fullpage}
\newcommand{\NN}{\mathbb{N}}
\newcommand{\ZZ}{\mathbb{Z}}

 \newcommand{\CC}{\mathbb{C}}
\newcommand{\DD}{\mathbb{D}}
\newcommand{\TT}{\mathbb{T}}
\newcommand{\EE}{\mathbb{E}}
\newcommand{\PP}{\mathbb{P}}
\newcommand{\RR}{\mathbb{R}}

\def\bb{\begin{color}{blue}}
\def\bg{\begin{color}{green}}
\def\br{\begin{color}{red}}
\def\eg{\end{color}}
\def\er{\end{color}}
\def\eb{\end{color}}

\let \leq \leqslant

\let \ge \geqslant
\let \geq \geqslant

\newcommand{\eps}{\epsilon}

\newcommand{\IIm}{\text{Im} \,}

\newcommand{\g}{\gamma}
\newcommand{\tg}{\tilde{\gamma}}
\newcommand{\tG}{\tilde{\Gamma}}

\newcommand{\dtt}{\, \dtt}

\newcommand{\capc}{\mathbf{c}}

\newcommand{\parsig}{{\boldsymbol \sigma}}

\begin{document}

\bibliographystyle{alpha}

\title{Small particle limits in a regularized Laplacian random growth model}
\author{Fredrik Johansson Viklund
\thanks{fjv@math.columbia.edu}}
\affil{Department of Mathematics, Columbia University}
\author{Alan Sola
\thanks{a.sola@statslab.cam.ac.uk}}
\affil{Statistical Laboratory, Centre for Mathematical Sciences, University of Cambridge}
\author{Amanda Turner
\thanks{a.g.turner@lancaster.ac.uk}}
\affil{Department of Mathematics and Statistics, Lancaster University}
\maketitle

\begin{abstract}
We study a regularized version of Hastings-Levitov planar random growth that models clusters formed by the aggregation of diffusing particles. In this model, the growing clusters are defined in terms of iterated slit maps whose 
capacities are given by 
\[c_n=\capc|\Phi_{n-1}'(e^{\parsig+i\theta_n})|^{-\alpha}, \quad \alpha\geq0,\] 
where $\capc>0$ is the capacity of the first particle, $\{\Phi_n\}_n$ are the composed 
conformal maps defining the clusters of the evolution, $\{\theta_n\}_n$ are independent uniform angles determining the positions at which particles are attached, and $\parsig>0$ is a regularization parameter which we take to depend on $\capc$. 
We prove that under an appropriate rescaling of time, in the limit as $\capc \to 0$, the clusters converge to 
growing disks with deterministic capacities, provided that $\parsig$ does not converge to $0$ too fast. We then establish scaling limits for the harmonic measure flow over longer time periods showing that, by letting $\alpha \to 0$ at different rates, this flow converges to either the Brownian web on the circle, a stopped version of the Brownian web on the circle, or the identity map. As the harmonic measure flow is closely related to the internal branching structure within the cluster, the above three cases intuitively correspond to the number of infinite branches in the model being either 1, a random number whose 
distribution we obtain, or unbounded, in the limit as $\capc \to 0$.

We also present several findings based on simulations of the model with parameter choices not covered by our rigorous analysis.
\end{abstract}
\tableofcontents

\section{Introduction}\label{intro}

Hastings-Levitov processes, introduced in \cite{HL98}, provide a framework for modelling planar random growth which occurs through the repeated aggregation of particles. Specific examples of such processes include off-lattice versions of diffusion-limited aggregation (DLA) \cite{WS81}, dielectric breakdown \cite{NPW86}, and the Eden model \cite{E61} for biological growth. In these models, individual particles are represented as conformal mappings, and the process of aggregation corresponds to repeated composition of such maps.

For $\capc > 0$, consider  the conformal map
\[f_{\capc}\colon \Delta=\{z\in \CC\colon |z|>1\}\rightarrow D_1=\Delta\setminus(1, 1+d],\]
that satisfies $f_{\capc}(z)=e^{\capc}z+\mathcal{O}(1)$ at infinity, and sends the exterior unit disk onto the exterior unit disk minus a slit of length $d=d(\capc)$. This map corresponds to a single particle (represented by the slit) being attached to the unit disk at the point 1. An explicit expression 
for these slit maps can readily be obtained.\footnote{The Hastings-Levitov model can be set up for more general particle shapes, but having a concrete expression for the building blocks allows 
us to conveniently perform certain computations explicitly. We believe our results hold for a wide class of reasonable particles.}
The length or ``size'' of the slit and the associated 
capacity increment $\capc$ are related via
\begin{equation}
e^{\capc}=1+\frac{d^2}{4(1+d)}.
\label{capsizerel}
\end{equation}
 Typically we are interested in understanding the geometry of the clusters as the size of the particles tends to zero, whilst the number of particles attached becomes very large. From the above expression it is easy to show that $d \asymp \capc^{1/2}$ as $\capc \rightarrow 0$. Therefore, the requirements that $\capc \to 0$ and $d \to 0$ are equivalent.

Let $\{\theta_k\}_{k=1}^{\infty}$ be a sequence of angles in $[0,2\pi)$ and let $\{c_k\}_{k=1}^{\infty}$ be a sequence of positive numbers. We introduce 
rescaled and rotated conformal maps
\begin{equation}
f_k(z)=e^{i \theta_k}f_{c_k}(e^{-i\theta_k}z).
\label{detmaps}
\end{equation}
These maps are then used as building blocks for a 
growth process defined using the iterated maps
\begin{equation}
\Phi_n(z)=f_1 \circ \cdots \circ f_n(z), \quad n=1,2,\ldots \,.
\label{Phidef}
\end{equation}
Each $\Phi_n$ is a conformal map of the exterior disk onto the complement 
of a compact set $K_n$, 
\[\Phi_n \colon \Delta \rightarrow D_n=\CC\setminus K_n,\]
and $\Phi_n$ has an expansion at infinity of the form 
\[\Phi_n(z)=e^{C_n}z+\mathcal{O}(1), \quad \textrm{where}\quad 
C_n=\sum_{k=1}^nc_k.\]
The sets $\{K_n\}_{n=1}^{\infty}$ form an increasing sequence and are referred to as 
the growing clusters, and $\mathrm{cap}(K_n)=e^{C_n}$ is the logarithmic 
capacity of the $n^{th}$ cluster and is comparable to the diameter of the cluster.

By choosing different sequences of angles and capacities, a wide class of 
growth processes can be described. In the present context, we are 
specifically interested in modelling the aggregation of diffusing particles: 
the point 
of attachment of each particle should then be determined by the hitting 
probability of a Brownian 
motion started at infinity. To achieve 
this, we choose the sequence $\{\theta_k\}_{k=1}^{\infty}$ to be 
independent random variables, each uniformly distributed on $[0, 2\pi)$. 
The conformal invariance of harmonic 
measure in the plane then means that at step $n+1$, a point on the 
boundary of the cluster $K_n$ is chosen according to harmonic 
measure seen from $\infty$, and a particle, the image under 
$\Phi_n$ of $(1,1+d_{n+1}]$, is attached at this point. Here, and in what 
follows, $d_{n+1}$ is short-hand notation for $d(c_{n+1})$.

From the point of view of physical systems, 
for instance in DLA, it is natural to request that all 
particles have roughly the same size after attachment to the growing cluster.
The final arc length of the $(n+1)^{th}$ arrival is given by 
\begin{equation}
\ell(d_{n+1})=\int_{1}^{1+d_{n+1}}|\Phi'_{n}(re^{i\theta_{n+1}})|dr =d_{n+1}|\Phi'_{n}(r_0e^{i\theta_{n+1}})|
\label{arclength}
\end{equation}
for some $r_0\in [1,1+d_{n+1}]$, 
assuming sufficient regularity of $|\Phi'_{n}|$ close to the boundary. 
In the Hastings-Levitov growth model with parameter $\alpha \in[0,2]$, 
usually referred to as $\mathrm{HL}(\alpha)$,
one sets $d_1=d$ and
\begin{equation}
d_{n+1}=\frac{d}{|\Phi'_{n}(e^{i\theta_{n+1}})|^{\alpha/2}}, \quad 
n=2, 3 ,\ldots.
\label{HLa}
\end{equation} 
This means that the size of every particle is scaled by 
a power of the derivative so as to take into account, to 
an extent that varies with $\alpha$, the local distortion associated with the 
conformal map $\Phi_n$ near the point of attachment. In particular, the choice
$\alpha=2$ means that, heuristically, $\ell(d_{n})$ is close to $d$ 
for each $n$ and this model was therefore proposed 
as a candidate for off-lattice DLA. 
Setting $\alpha=0$ means that $d_{n}=d$ for all $n$, 
and so the maps $\{f_n\}_{n=1}^{\infty}$ are independent and identically distributed. Although this is the least physical of the Hastings-Levitov models in that the factor by which the particles are distorted can be shown to grow exponentially fast, it is the mathematically
most tractable, and has been studied 
in \cite{RZ05} and \cite{NT12}. 

The $\mathrm{HL}(\alpha)$ models, for $\alpha>0$, lend themselves well to 
computer simulations (see for instance \cite{HL98} and \cite{Davetal99}); however it 
seems that it is hard to establish rigorous results concerning their 
long-time behavior. Looking at the definition \eqref{sigregdef}, it 
becomes clear that the growth 
of clusters at any given stage depends on its past history in a 
complicated way. Furthermore, from a technical point of view, 
the derivatives $\Phi_n'$ become badly behaved on the boundary as $n$
becomes large, making it difficult to obtain useful estimates on 
$\{d_n\}$.
 
In this paper, we study a regularized version of the Hastings-Levitov model for $\alpha>0$, which we call $\mathrm{HL}(\alpha, \parsig)$. In this model, the capacity increments are given by
\begin{equation}
c_{n+1}=\frac{\capc}{|\Phi_{n}'(e^{\parsig+i\theta_{n+1}})|^{\alpha}}, \quad 
n=0,1,2,\ldots.
\label{sigregdef}
\end{equation}
Here $\parsig >0$ is a regularization parameter that is 
allowed to depend on $\capc$ and we take $\Phi_0$ to be the identity map. The 
geometric quantities $d_n$ can then be determined using
\eqref{capsizerel}. Similar conformal mapping models of 
growth phenomena also appear in \cite{CM02} and \cite{RZ05}: Carleson and Makarov introduced and studied the 
analogous regularization for a Loewner chain formulation of the Hele-Shaw 
flow, see Section~2.3 of \cite{CM01} and this 
provided the inspiration for studying
the particular regularization considered in this paper. We have also 
benefited from the work of Rohde and Zinsmeister (see in particular 
\cite[Section 6]{RZ05}). 

Letting $\parsig \to 0$, we recover
an equivalent formulation of $\mathrm{HL}(\alpha)$. In view of 
\eqref{arclength}, a natural range of the parameter is 
$\parsig \asymp d$. The range $\parsig\asymp d$ seems especially interesting since direct calculation shows that $|f_{\capc}'(e^{\parsig + i \theta})|$ is scale invariant for $\parsig$ and $\theta$ in this range and, furthermore, simulations reveal non-trivial geometric structures.
For all $\parsig < \infty$, 
each $c_n$ still depends on the whole sequence $\theta_1,\ldots, \theta_n$ via the scaling by the derivative, and so the model retains the long-term dependencies featuring in $\mathrm{HL}(\alpha)$. However, provided $\parsig$ does not tend to $0$ too quickly, the scaling is less singular than for $\parsig=0$ and allows us to make use of distortion estimates on conformal maps. Figure \ref{HLpics} illustrates the effect of varying $\parsig$ on the $\mathrm{HL}(\alpha, \parsig)$ cluster in the cases when $\alpha=0.5$ and $\alpha=2$.

There are a number of well known open problems relating to $\mathrm{HL}(\alpha)$ for $\alpha>0$. Whilst we are not able to solve these problems directly, we are able to give partial answers to related questions for the regularized model. 

Hastings and Levitov predicted that, for small particles, the $\mathrm{HL}(\alpha)$ process undergoes a phase transition as $\alpha$
increases through the point $\alpha_{\textrm{crit}}=1$: for $\alpha \in [0,1)$, the clusters look like disks, whereas for $\alpha \in (1,2]$ the clusters seem to be random anisotropic shapes. Simulations of $\mathrm{HL}(\alpha, \parsig)$ suggest that this phase transition at $\alpha_{\textrm{crit}}=1$ is also present in the regularized model when $\parsig=\mathcal{O}(d)$ (see Figure \ref{HLpics}). We cannot at present prove any statements regarding
the existence or non-existence of limit clusters in this regime, 
or the possible presence of a sharp phase transition phenomenon.
However, we have been able to show that the macroscopic shape of the regularized $\mathrm{HL}(\alpha, \parsig)$ clusters is a disk for all values of $\alpha$, even if $\parsig \to 0$ as $\capc \to 0$, provided that this convergence is not too fast relative to that of $\capc$. 

As each particle in a cluster is attached to a single ``parent'' particle, the cluster can be viewed as a random tree, rooted at the vertex corresponding to the initial unit disk. The branches of the cluster are the connected components of the graph obtained by deleting the root vertex so each branch is a maximal set of particles sharing a common ancestor particle. It has been shown that the $\mathrm{HL}(0)$ cluster has a single infinite branch as the number of particles goes to infinity (see \cite{NT12} for a version of this result in the small particle limit, although it is known to be true for any particle size). A natural question is to ask what the smallest value of $\alpha$ is for which the $\mathrm{HL}(\alpha)$ cluster has more than one infinite branch, and whether the number of infinite branches increases as a function of $\alpha$. For the regularized $\mathrm{HL}(\alpha, \parsig)$ cluster, again provided that $\parsig$ does not converge to zero too fast, we show that if $\alpha$ is allowed to tend 
to zero appropriately, then the limiting harmonic measure flow is that of a cluster with a random number of infinite branches and the distribution of the number of infinite branches is stochastically increasing in $\alpha$.

Interpreting branching in terms of harmonic measure flow, our main results in this direction can be phrased as follows. Suppose that $\parsig \gg (\log \capc^{-1})^{-1/2}$. Then in the limit as $\capc \to 0$, the $\mathrm{HL}(\alpha, \parsig)$ cluster is a disk with internal structure consisting of
\begin{itemize}
\item one infinite branch if $\alpha \ll \capc^{1/2}$; 
\item a random number of infinite branches, whose distribution is stochastically increasing in $a$, if $\alpha \capc^{-1/2} \to a \in (0,\infty)$; 
\item deterministic radial growth if $\alpha \gg \capc^{1/2}$.
\end{itemize}

In the next section we give an overview of our main results as well as an outline of the proof strategies that we have used. In Section \ref{simul} we discuss simulations that we have performed, notation and preliminary estimates are in Section \ref{prelim}, and the proofs take up the remainder of the paper. Results for deterministic capacity sequences are in Section \ref{norristurner}, capacity limits are in Section \ref{capacityconvergence}, macroscopic scaling limits are in Section \ref{mapconvergence} and scaling limits of the harmonic measure flow are in Section \ref{flowcon}.
\subsection*{Acknowledgements}
The authors thank James Norris for many useful discussions and suggestions. We would also like to thank Sunil Chhita, Daniel Elton and Noam Berger for helpful conversations and advice, and Bati Sengul for 
calling the paper \cite{BLG} to our attention.

This work was initiated while the authors were attending the special semester 
on Complex Analysis and Integrable Systems at Institut Mittag-Leffler in 2011. 
The authors are grateful to the Institute and its staff for providing an 
ideal environment for mathematical research.

FJV was supported by the Simons Foundation, NSF grant DMS-1308476, Institut Mittag-Leffler, and the AXA Research fund. AS acknowledges support from the EPSRC under grant EP/103372X/1, Institut Mittag-Leffler, and the AXA Research fund. We thank the Statistical Laboratory at the University of Cambridge and the Department of Mathematics and Statistics at Lancaster University for their hospitality and financial support. 

\section{Overview of results}\label{overview}

Our first aim in this paper is to show that as 
$\capc\rightarrow 0$, in the scaling limit $N=\lfloor T/\capc\rfloor$, the $\mathrm{HL}(\alpha, \parsig)$ map $\Phi_{N}$ converges to a 
deterministic limit, namely, the map 
\[z\mapsto \Psi_T(z)=(1+\alpha T)^{1/\alpha}z, \quad z\in \Delta.\] 
Note that as $\alpha \to 0$, $\Psi_T(z) \to e^Tz$. This recovers the result in \cite{NT12}, where it was shown that
the corresponding scaling limit for $\mathrm{HL}(0)$ is the deterministic 
map 
\[z\mapsto e^Tz, \quad z \in \Delta.\]

Since our growth process exhibits a 
non-trivial dependence on its past, the time-reversal arguments 
in \cite{RZ05} and the martingale techniques in \cite{NT12} are not 
immediately applicable. However, we are able to show that, provided $\parsig$ does not tend to 0 too quickly, the capacities $\{c_k\}_{k=1}^{\infty}$ are close, with high probability, to the deterministic sequence $\{c^*_k\}_{k=1}^{\infty}$ defined by
\begin{equation}
c^*_k=\frac{\capc}{1+\alpha \capc(k-1)}, \ \ k=1,2,\ldots 
\label{detcaps}
\end{equation}
Some intuition behind the form of this sequence is provided at the end of this section. In Section \ref{capacityconvergence} we prove the following theorem, which is stated precisely with detailed hypotheses in Theorem 
\ref{capincrconv}. 

\begin{Mth}[Convergence of capacities]
Let $\parsig \gg (\log \capc^{-1})^{-1/2}$. Then there exists some absolute constant $\beta>0$ such that
\[
\mathbb{P} \left ( \sup_{n \leq N} \left | \log \frac{c_n}{c^*_n} \right | > \alpha \capc^{\beta} \right ) \to 0
\] 
as $\capc \to 0$.
\end{Mth}

Using the fact that $|1-x| \leq |\log x|(1+|\log x|/2)$ for small $x$, this is equivalent to showing that
\[|c_n-c^*_n|<2\alpha\capc^{1+\beta}\]
for the above range of $n$. 

We prove the above result in two stages. Firstly, we define a ``starred'' growth process by using the deterministic capacities \eqref{detcaps} in place of \eqref{sigregdef} in the constructions of the building blocks $\{f^*_k\}$ and in the corresponding compositions $\Phi^*_{n}$. We couple the starred and un-starred maps by using the same angles of rotation; 
however we suppress this dependence for notational simplicity. As the $\Phi^*_{n}$ maps do not exhibit the
complicated dependence of the $\Phi_{n}$ maps, in Section \ref{norristurner} we adapt the martingale techniques in 
\cite{NT12} to prove that these maps converge in probability as $\capc \to 0$ to a deterministic map that takes the exterior unit disk to the complement of an inflated disk. In fact, we show more generally that this result holds if $\{c^*_n\}$ is replaced by any deterministic sequence of capacities that are uniformly bounded by $\capc$.
Secondly, in Section \ref{capacityconvergence} we use the coupling of $\Phi'_n$ with $(\Phi^*_n)'$ to bound their difference in terms of the parameter $\parsig$, using a recursive 
argument. It is in this part of the proof that the regularization parameter $\parsig$ plays a crucial role, 
as the fact that we are evaluating derivatives away from the boundary of $\Delta$ allows us to make use of uniform distortion bounds on quantities like
$|\Phi'_n(e^{\parsig+i\Theta})|$. A Gr\"onwall-type argument completes the proof 
that $\{c_k\}$ and $\{c^*_k\}$ are uniformly close.

% In for instance, the processes
% \[(\Gamma_{n}(z)\colon n\geq 1), \quad z\in \Delta \quad \textrm{fixed},\]
% are no longer martingales (as was the case in \cite{NT12}), 
% and the law of the maps
% \[\Phi_n=f_1\circ\cdots \circ f_n\]
% is not invariant under reversal of the order of composition 
% (unlike \cite[Section 4]{RZ05}), and so the previously known techniques do not apply.

By adapting the techniques in \cite{JST12}, we next show that the conformal 
maps $\Phi_N$ converge in probability to the deterministic limit map 
$\Psi_T(z)=(1+\alpha T)^{1/\alpha}z$, provided we scale the number of 
arriving particles as $N = \lfloor T/ \capc \rfloor$.
Our arguments here rely on the facts that the 
convergence $\sup_k|c_k-c^*_k|/\capc \rightarrow 0$ implies weak 
convergence of driving 
measures for the Loewner representation of the growth process, and that weakly 
convergent driving measures lead to sequences 
of conformal maps that converge in the sense of Carath\'eodory. 
In Section \ref{mapconvergence} 
this then leads to the first of our two main theorems, which is stated precisely with detailed hypotheses in Theorem 
\ref{MainThm1}.

\begin{Mth}[Convergence of clusters to disks]
Let $T>0$ and $\alpha > 0$ be fixed. 
Set $N=\lfloor T/\capc\rfloor$ and suppose
$\parsig \gg (\log \capc^{-1})^{-1/2}$.
Then, as $\capc\to 0$, the laws of the maps $\Phi_{N}$ converge weakly with respect to uniform convergence on compact subsets to a point mass at $\Psi_T(z)=(1+\alpha T)^{1/\alpha}z$.
%, viewed as random 
%measures on the space of normalized conformal maps equipped 
%with uniform convergence on compact subsets, converge in distribution 
%to a point mass at $\Psi_T(z)=(1+T)^{1/\alpha}z$.
\end{Mth}

%\begin{Mth}[Convergence of clusters to disks]
%Let $T>0$ and $\alpha \in (0, \infty)$ be fixed. 
%Set $N=\lfloor T/(\alpha \capc)\rfloor$, suppose
%$\parsig \gg (\log \capc^{-1})^{-1/2}$, and define
%\[\Phi_{N(\eps)}=f_{c_1}\circ\cdots \circ f_{c_{N(\eps)}},\]
%with $\{c_k\}_{k=1}^{N(\eps)}$ given by \eqref{sigregdef}.
%Then, as $\capc\to 0$, the laws of the maps $\Phi_{N(\eps)}$, viewed as random 
%measures on the space of normalized conformal maps equipped 
%with uniform convergence on compact subsets, converge in distribution 
%to a point mass at $\Psi_T(z)=(1+T)^{1/\alpha}z$.
%\end{Mth}

Observe that when $\alpha=2$ this result implies that the cluster after $N$ arrivals is approximately a disk with area $\pi (1 + 2 \capc N)$, so the area increases at a constant rate in the number of particles. This is consistent with the idea that $\alpha=2$ corresponds to a model in which all particles are the same size. Even though our choice of $\parsig$ is not in the range needed to ensure that the sizes of the particles are not distorted by the composition process, namely $\parsig \leq d$, this result suggests, nevertheless, that a deterministic correction to the capacities is enough to ensure that {\em on average} all particles are the same size. 
Similarly, when $\alpha=1$, the total 
boundary length of a cluster with $N$ particles is close to $2\pi (1+\capc N)$, and thus grows at a constant 
rate. This is the instance of the Hastings-Levitov model thought to be an analog of the Eden 
model and so the above result is consistent with a model which exhibits growth proportional to local arclength.  

We next prove 
that the rescaled harmonic measure flow converges in a certain space of weak 
flows introduced in \cite{NT12}, and we identify the weak limit. The harmonic measure flow exhibits a phase transition at $\alpha=0$ and we therefore consider different ``off-critical'' limits by letting $\alpha \to 0$ at different rates and obtain the following result. 

\begin{Mth}[Convergence of harmonic measure flow]
Suppose that $\parsig \gg (\log \capc^{-1})^{-1/2}$. Then as $\capc\to 0$, on timescales of order $\capc^{-3/2}$, one of the following three situations arises.
\begin{itemize}
\item If $\alpha \capc^{-1/2} \to 0$, the harmonic measure flow converges to the Brownian web.
\item If $\alpha \capc^{-1/2} \to \infty$ (sufficiently slowly), the harmonic measure flow converges to the identity flow.
\item If $\alpha \capc^{-1/2} \to a \in (0,\infty)$, the harmonic measure flow converges to a time-change of the Brownian web, stopped at a finite time that is decreasing in $a$. 
\end{itemize}
\end{Mth}

The proof relies on the fact that 
$\sup_k|c_k-c^*_k|/\capc\to 0$ and 
that the sequence $\{c^*_k\}$ is independent of the history of the cluster. This enables us to show that the harmonic measure flow is close to a martingale and identify the Brownian limits over time periods of order $\capc^{-3/2}$. Full statements and proofs are given in Theorems \ref{ptwiseflowconvergence} and 
\ref{flowconvergence}. We note that the result is not particular to the sequence $\{c^*_k\}$, and 
a similar classification holds for general particle sequences that do 
not exhibit a dependence on the past history of the growth. This phenomenon is illustrated in Figure \ref{BWflowpics}.

\begin{figure}[h!]
    \subfigure[$\alpha=10^{-4}$]
      {\includegraphics[width=0.4 \textwidth]{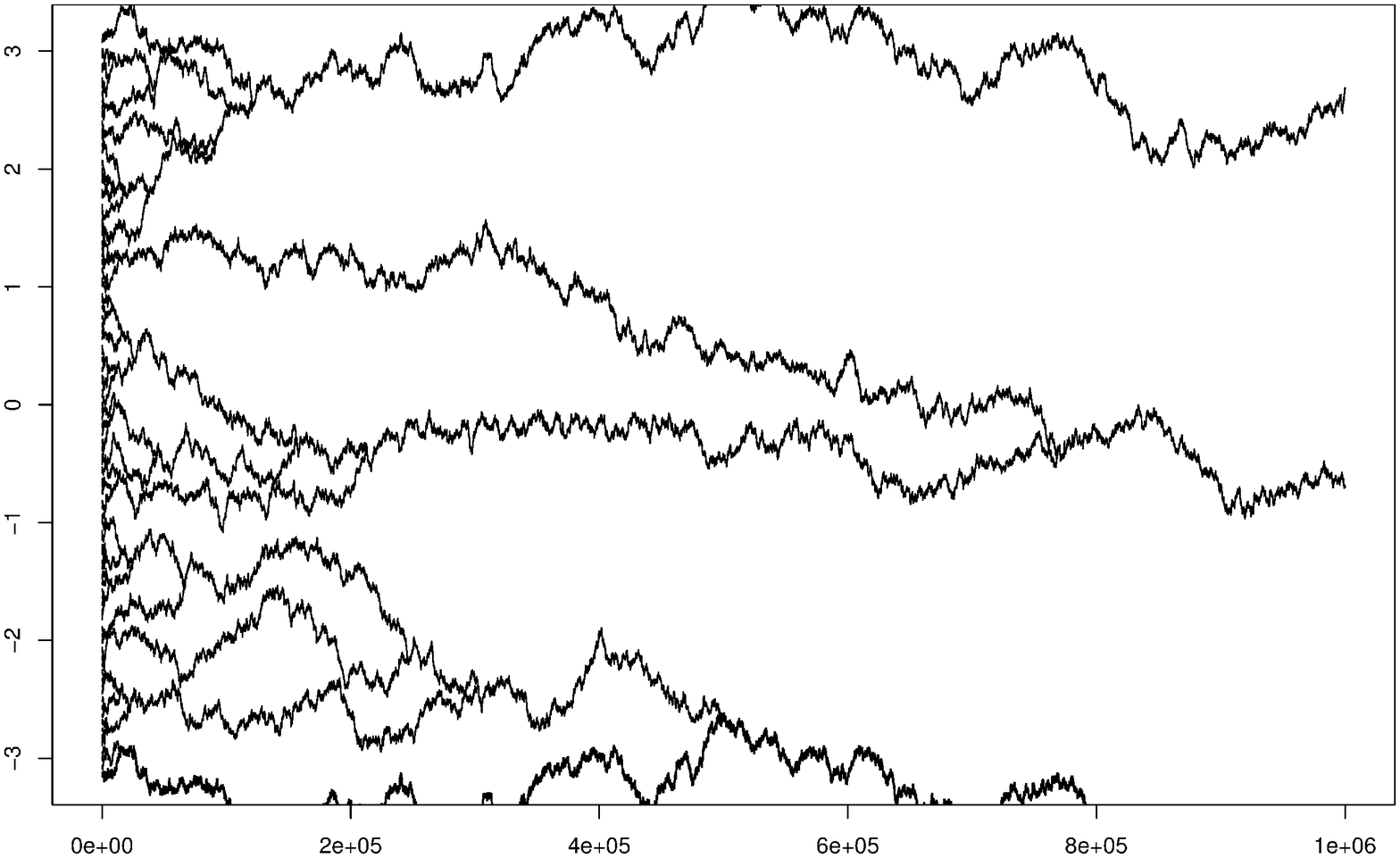}}
    \hfill
    \subfigure[$\alpha=10^{-2}$]
      {\includegraphics[width=0.4 \textwidth]{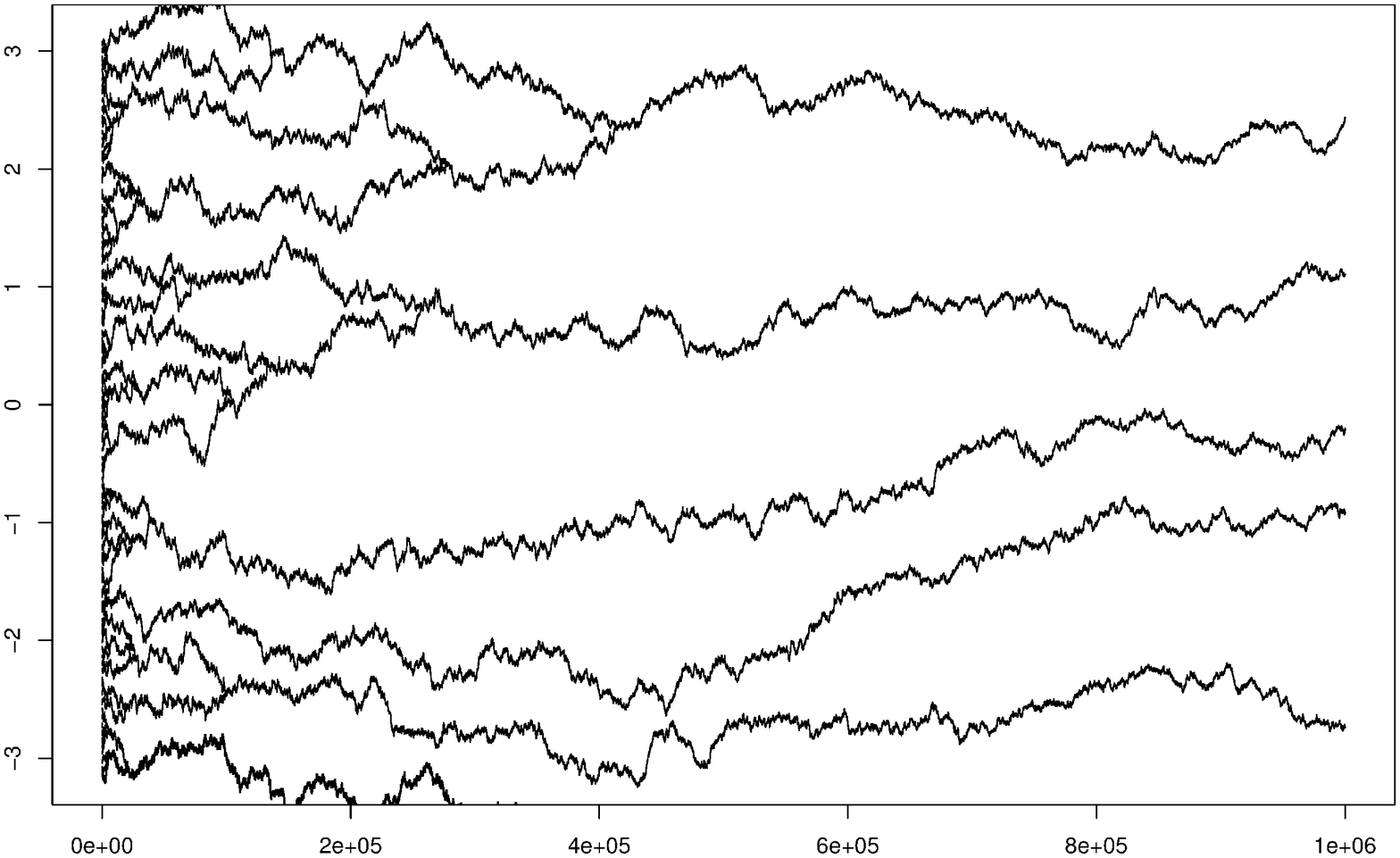}}
  \caption{\textsl{Harmonic measure flow corresponding to $\Phi^*_n$ with $1,000,000$ particles when $\capc=10^{-4}$ for $\alpha=10^{-4} = \capc$ (left) and $\alpha=10^{-2}=\capc^{1/2}$ (right). }}
  \label{BWflowpics}
\end{figure}

It is shown in \cite{NT12} that the evolution of harmonic measure on the cluster boundary is closely related to the random tree structure that is contained within the cluster.
Intuitively there is a correspondence between the number of infinite branches in a cluster and the number of distinct paths arising from the harmonic measure flow started at time zero  that survive infinitely long. As all Brownian motions on the circle starting at a fixed time eventually coalesce into a single Brownian motion, we interpret this result as the $\mathrm{HL}(\alpha, \parsig)$ cluster when $\alpha \ll \capc^{1/2}$ having a single infinite branch, or equivalently all particles arriving beyond a certain time sharing a common ancestor. If $\alpha \gg \capc^{1/2}$, the harmonic measure flow converges to the identity flow which intuitively corresponds to the number of infinite branches becoming unbounded in the limit as $\capc \to 0$. If 
$\alpha\capc^{-1/2}\to a \in (0,\infty)$, the Brownian web stopped at a finite time intuitively corresponds to a random number 
of infinite branches in the
$\mathrm{HL}(\alpha, \parsig)$ cluster. By using a result of Bertoin and Le Gall \cite{BLG} that relates the Brownian web to Kingman's coalescent, we are able to give the distribution of this random number and show that it is stochastically increasing in $a$.

% It is perhaps somewhat surprising that we observe a a limit that is 
% non-deterministic, yet different from the Brownian Web that appear in 
% $\mathrm{HL}(0)$. (\br Wishful thinking...\er) 
% \begin{theorem}
% Flow convergence?
% \end{theorem}
% The $\mathrm{HL}(\alpha, \parsig)$ model thus exhibits an interesting 
% anisotropic behavior on the level of flows that is not present when 
% $\alpha=0$, and is not readily visible in pictures of the clusters.

\subsection*{Regularization at infinity}
We give an example of a simple regularization for which it is possible to calculate the capacity sequence $\{c_k\}$ explicitly. This provides some intuition behind the definition of the deterministic sequence $\{c^*_k\}$ in \eqref{detcaps} which features in our coupling arguments.

In the definition of $\textrm{HL}(\alpha,\parsig)$ let $\parsig=\infty$, that is, let $\parsig \to \infty$. Then at the $k^{th}$ step of the process, we are scaling 
successive capacity increments by the derivatives of the conformal maps at infinity, that is, by the total capacity of the growing cluster after $k-1$ steps. 
For a single building 
block then, $|f'_{\capc}(\infty)|^{-\alpha}=\exp(-\alpha \capc)$.
% so that
% \[|f'_{\capc}(\infty)|^{-\alpha}=1-\alpha \capc + \mathcal{O}(\capc^2), \quad \capc\rightarrow 0.\]
Let $q(k)=c_k/\capc$. Since the point at infinity is fixed by the individual maps $f_{c_k}$, the chain rule yields
\[q(k) = |\Phi_{k-1}'(e^{\parsig+i\theta_k})|^{-\alpha}=\prod_{j=1}^{k-1}|f'_{c_j}(\infty)|^{-\alpha}
=e^{-\alpha \capc \sum_{j=1}^{k-1}q(j)}.\]
% =\capc \prod_{j=1}^{k-1}\left(1-\alpha \capc q(j) +\mathcal{O}\left(\capc^2\right)\right),\]
We know \emph{a priori} that $q(k)\leq 1$ for all $k$ since the total capacity of the cluster is non-decreasing.
%and this leads to a uniform estimate on the errors in terms of $\capc^2$. 
From the relation
\[q(k)-q(k-1)=e^{-\alpha \capc \sum_{j=1}^{k-1}q(j)}
\left(e^{-\alpha \capc q(k-1)}-1\right),\]
% \begin{align*}
% \log[q(k)]-\log[q(k-1)]&=\log\left(1-\alpha \capc q(k-1))
% +\mathcal{O}(\capc^2 \right)\\&=-\alpha \capc q(k-1)+\mathcal{O}(\capc^2).
% \end{align*}
we obtain
\[\frac{q(k)-q(k-1)}{q(k-1)}=e^{-\alpha \capc q(k-1)}-1=-\alpha\capc q(k-1)+\mathcal{O}(\capc^2),\]
and upon dividing by $q(k-1)$ and integrating over $[1,n]$, we recover
\[\left[-\frac{1}{q(x)}\right]_{1}^n=-\alpha \capc(n-1)+n\mathcal{O}(\capc^2).\]
By assumption, $q(1)=1$, and after rearranging, we find
\[c_n = \capc q(n)=\capc (1+\alpha \capc (n-1)+n\mathcal{O}(\capc^2))^{-1}.\] 
When we pass to the limit with the natural scaling 
$N = \lfloor T/\capc \rfloor$, the last equation produces 
the expressions in \eqref{detcaps}
 as leading terms. (Cf. \cite[p.251]{HL98}.)

% \begin{figure}
% \centering
% \includegraphics[height=0.3\textheight,width=0.4\textwidth]{prepartHL(0.5,0.2)_d02.eps}
% \includegraphics[height=0.3\textheight,width=0.4\textwidth]{prepartHL(0.5,d)_d02.eps}
% \caption{Samples of capacity increment sequences in $\textrm{HL}(0.5,0.2)$ and
% $\textrm{HL}(0.5,0.02)$, with $\sqrt{\capc}=0.02$ and $N(\eps)=25,000$.}
% \end{figure}

%% Reorganized //FJV

\section{Simulations}\label{simul}

In this section we show simulations that illustrate the results above. We also present some experimental findings for the model with the parameter choice $\parsig = d$ that is not covered by our rigorous analysis. Simulations of clusters were generated using computer code that is based on code available on C. McMullen's webpage \cite{McMuProg}.
A repository of simulations that includes values of $\alpha$ and $\parsig$ not shown below can be found in \cite{TurnerWeb}.

\begin{figure}[ht!]
    \subfigure[$\mathrm{HL}(0.5, 1)$]
      {\includegraphics[width=0.4 \textwidth]{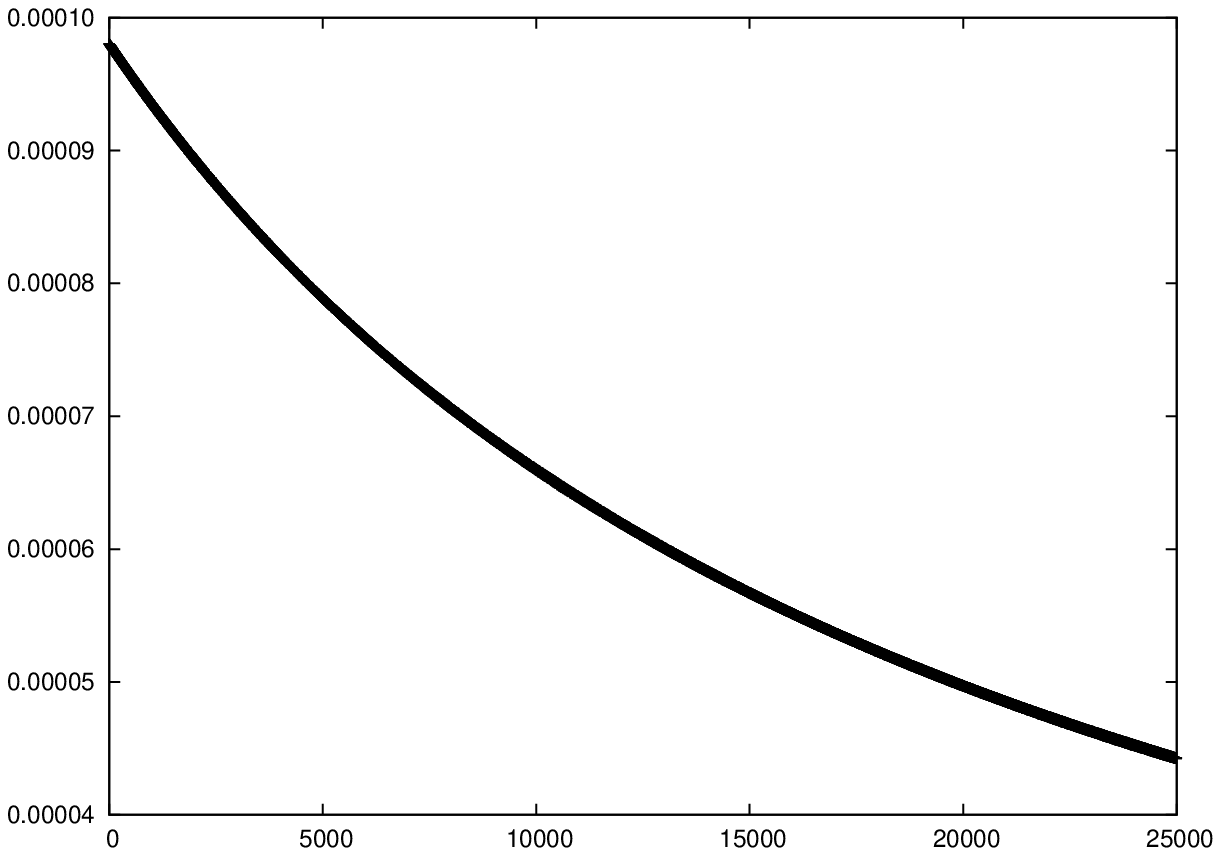}}
    \hfill
    \subfigure[$\mathrm{HL}(2, 1)$]
      {\includegraphics[width=0.4 \textwidth]{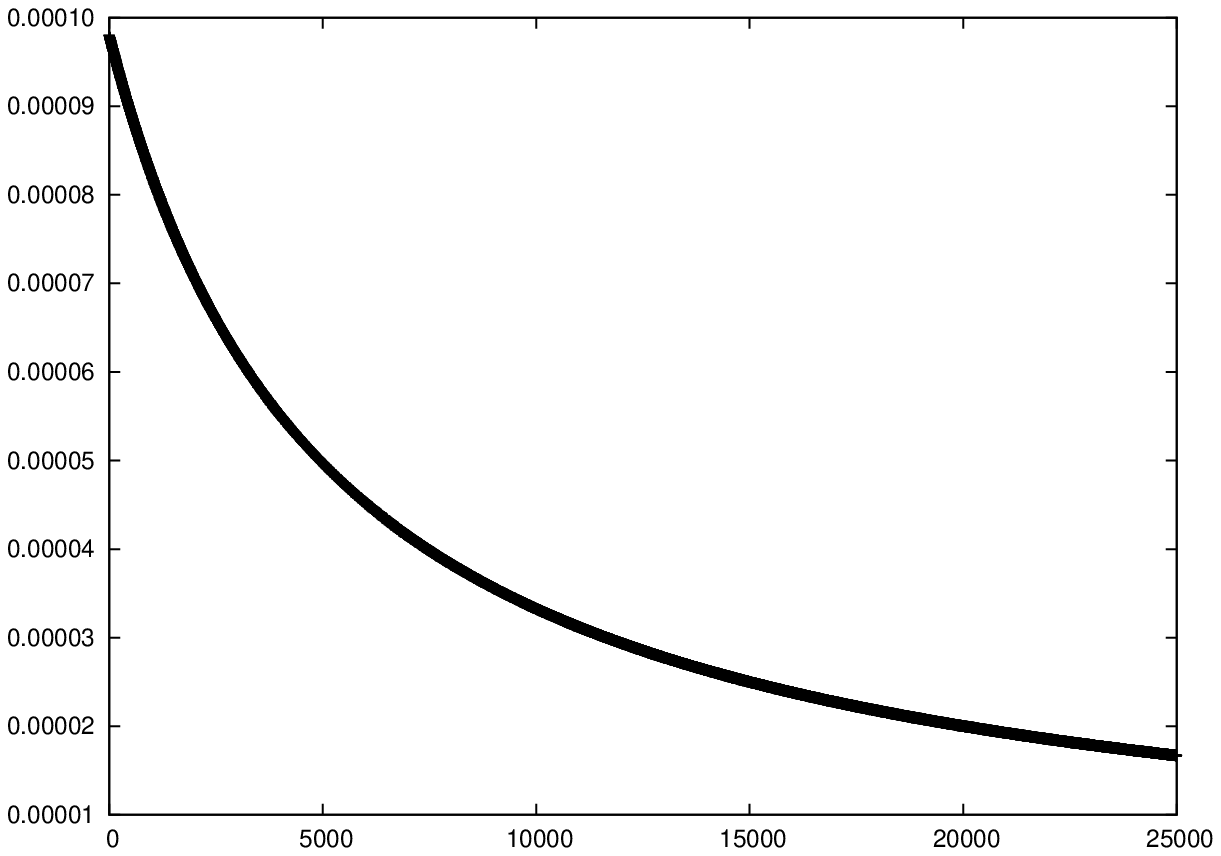}}
		\\
    \subfigure[$\mathrm{HL}(0.5, 0.2)$]
      {\includegraphics[width=0.4 \textwidth]{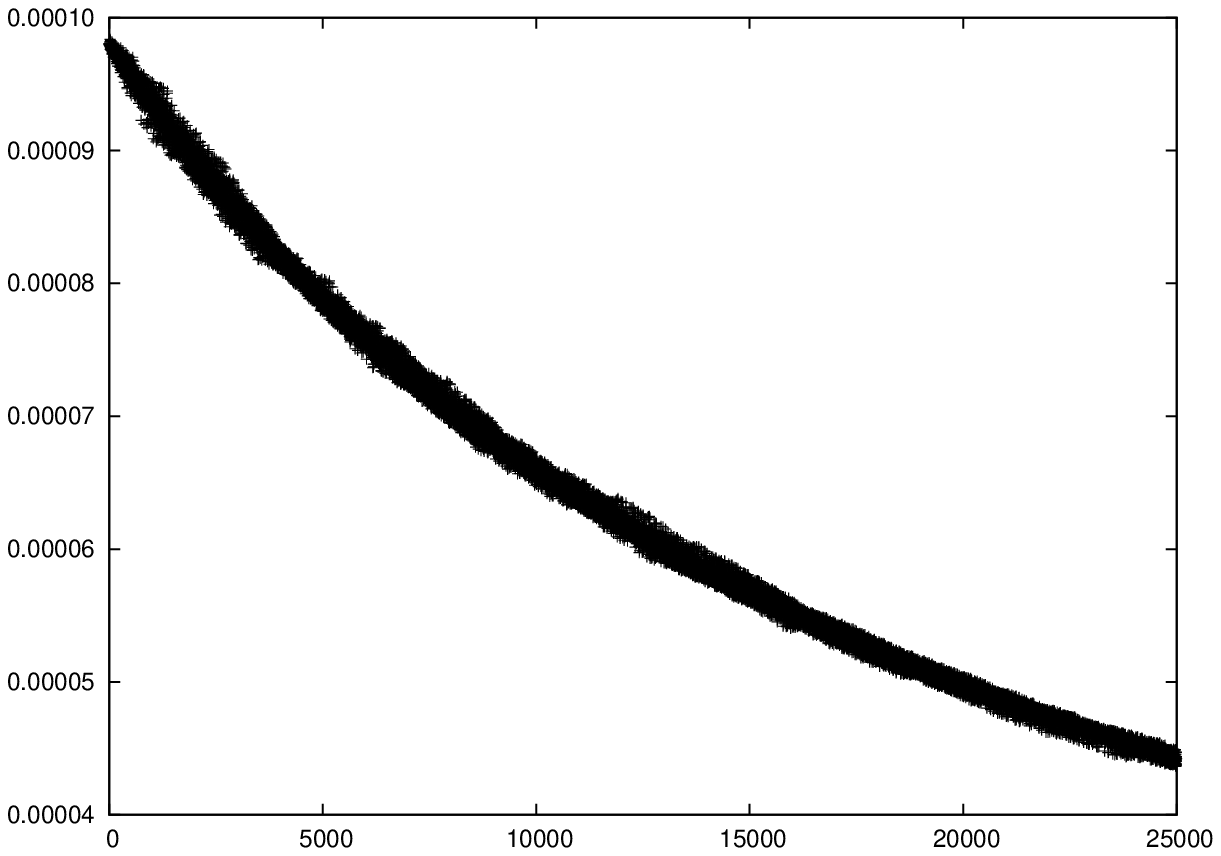}}
    \hfill
    \subfigure[$\mathrm{HL}(2, 0.2)$]
      {\includegraphics[width=0.4 \textwidth]{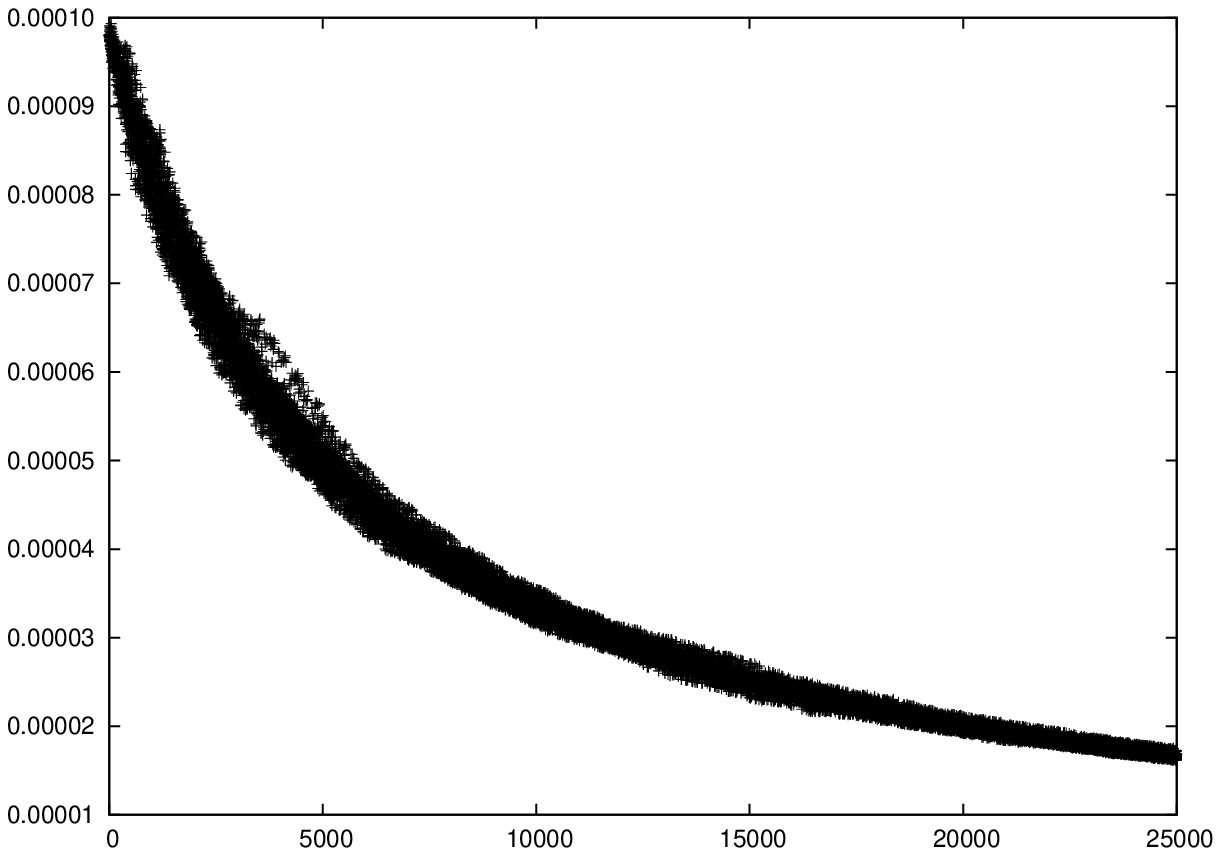}}
    \\
		\subfigure[$\mathrm{HL}(0.5, 0.02)$]
      {\includegraphics[width=0.4 \textwidth]{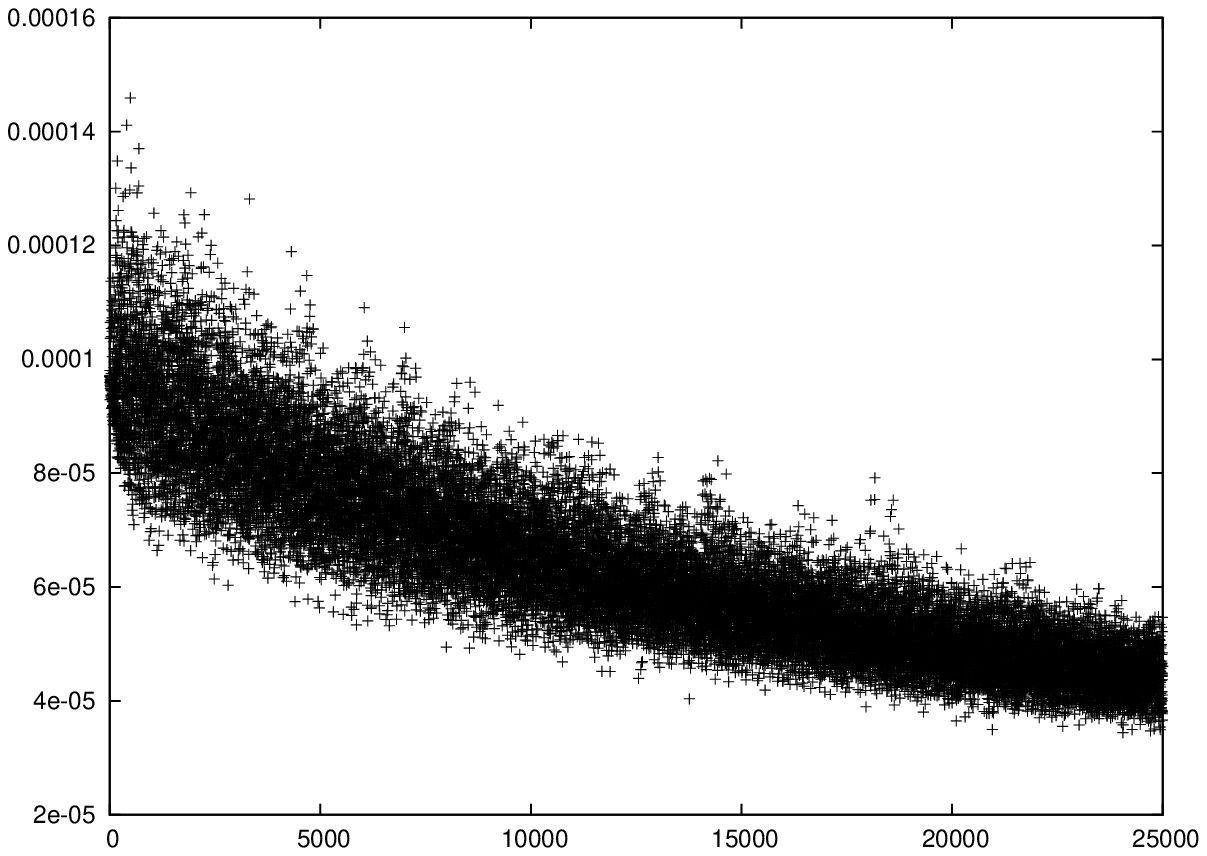}}
    \hfill
    \subfigure[$\mathrm{HL}(2, 0.02)$]
      {\includegraphics[width=0.4 \textwidth]{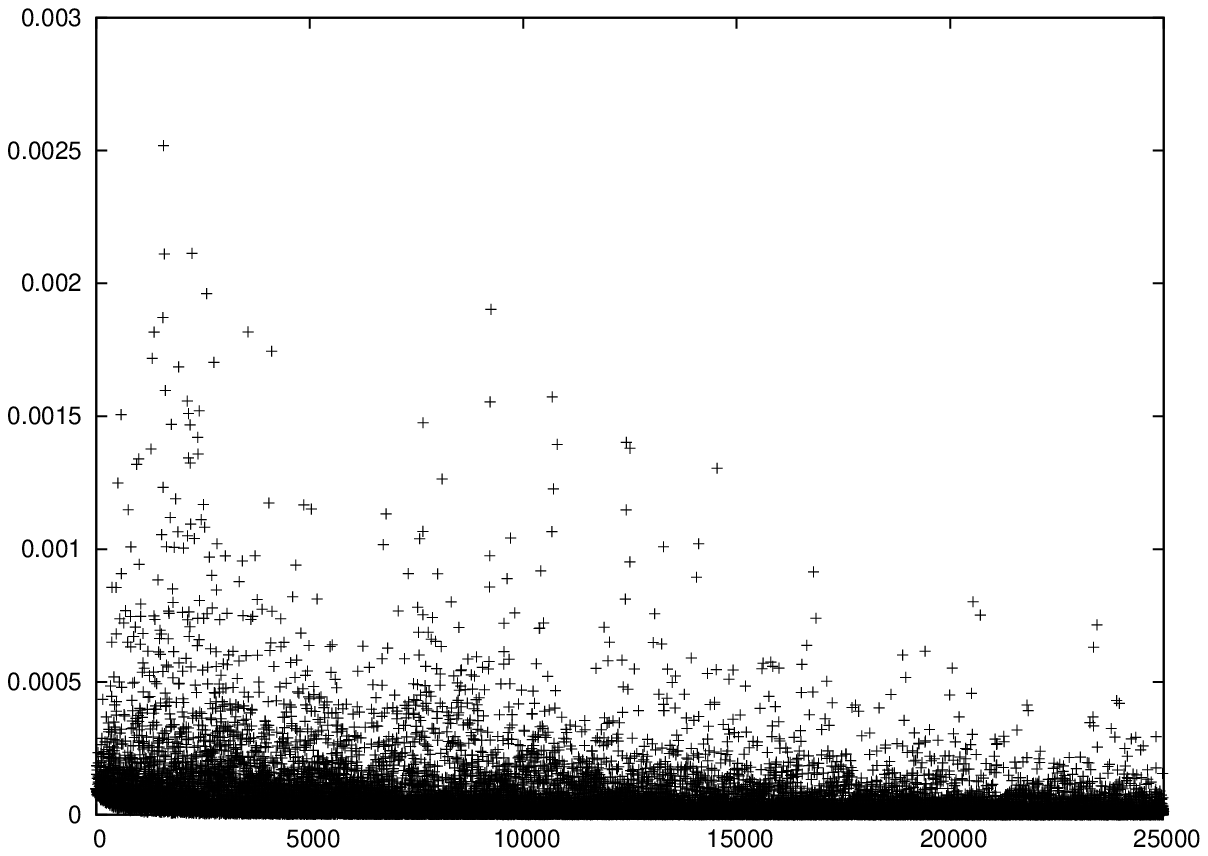}}
  \caption{\textsl{Capacity sequences $\{c_n\}_{n=1}^{25,000}$ for $\mathrm{HL}(\alpha, \parsig)$ with $\capc=10^{-4}$ for $\alpha=0.5$ (left) and $\alpha=2$ (right) when $\parsig=1$, $\parsig=0.2=2(\log \capc^{-1})^{-1/2}$ and $\parsig=0.02=d$.}}
  \label{capacityfig}
\end{figure}

Figure \ref{capacityfig} shows samples of capacity sequences 
$\{c_k\}_{k=1}^{25,000}$ for $\alpha=0.5$ and $\alpha=2$ for three different values of 
the regularization parameter $\parsig$. 
When $\parsig$ is sufficiently large compared to $\mathbf{c}$, the sequence $\{c_k\}$ is essentially indistinguishable from the 
deterministic sequence $\{c^*_k\}$ as predicted by Theorem \ref{capincrconv}; when $\parsig$ is on the order of the 
basic particle size $d$, the overall trend in $\{c_k\}$ 
follows the deterministic sequence for $\alpha =0.5$, 
but the random fluctuations increase in amplitude, whereas for $\alpha=2$ there is no clear deterministic limit for $\{c_k\}$ as $\capc\to 0$. Indeed, we do not believe that there is a deterministic limit when $\alpha>1$ if $\parsig$ is chosen small enough (say $\parsig = \mathcal{O}(\capc^{1/2})$), but 
we do not have a proof of this. Simulations further suggest that it should be possible to strengthen Theorem \ref{capincrconv}, and consequently all subsequent results, to hold for $\parsig \gg \capc^{\gamma}$ for some $\gamma \geq 1/2$. However, the distortion estimates that we use in our proof are no longer sufficient in this case.  

\begin{figure}[h!]
    \subfigure[$\mathrm{HL}(0.5, 1)$]
      {\includegraphics[width=0.33 \textwidth]{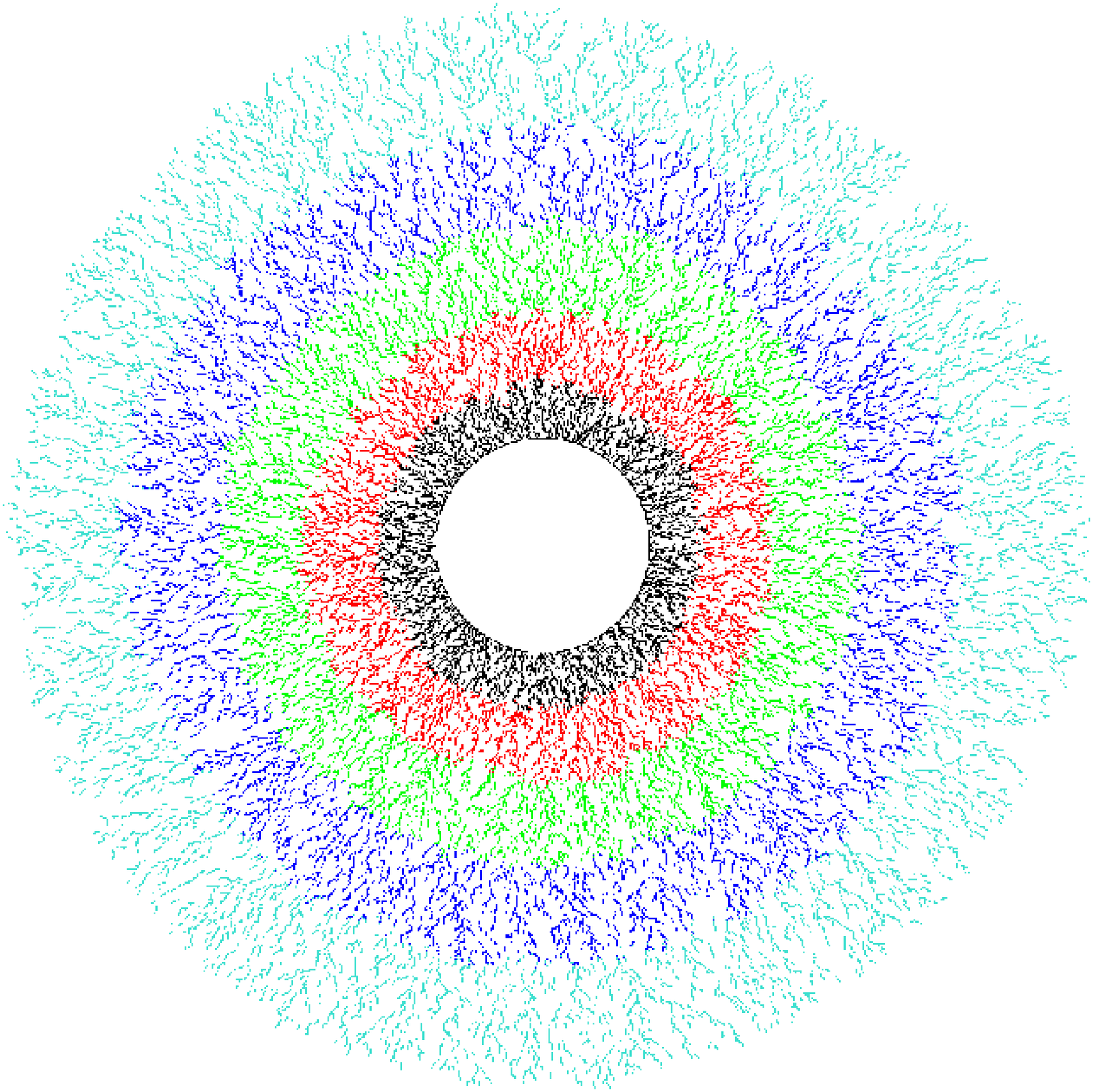}}
    \hfill
    \subfigure[$\mathrm{HL}(2, 1)$]
      {\includegraphics[width=0.33 \textwidth]{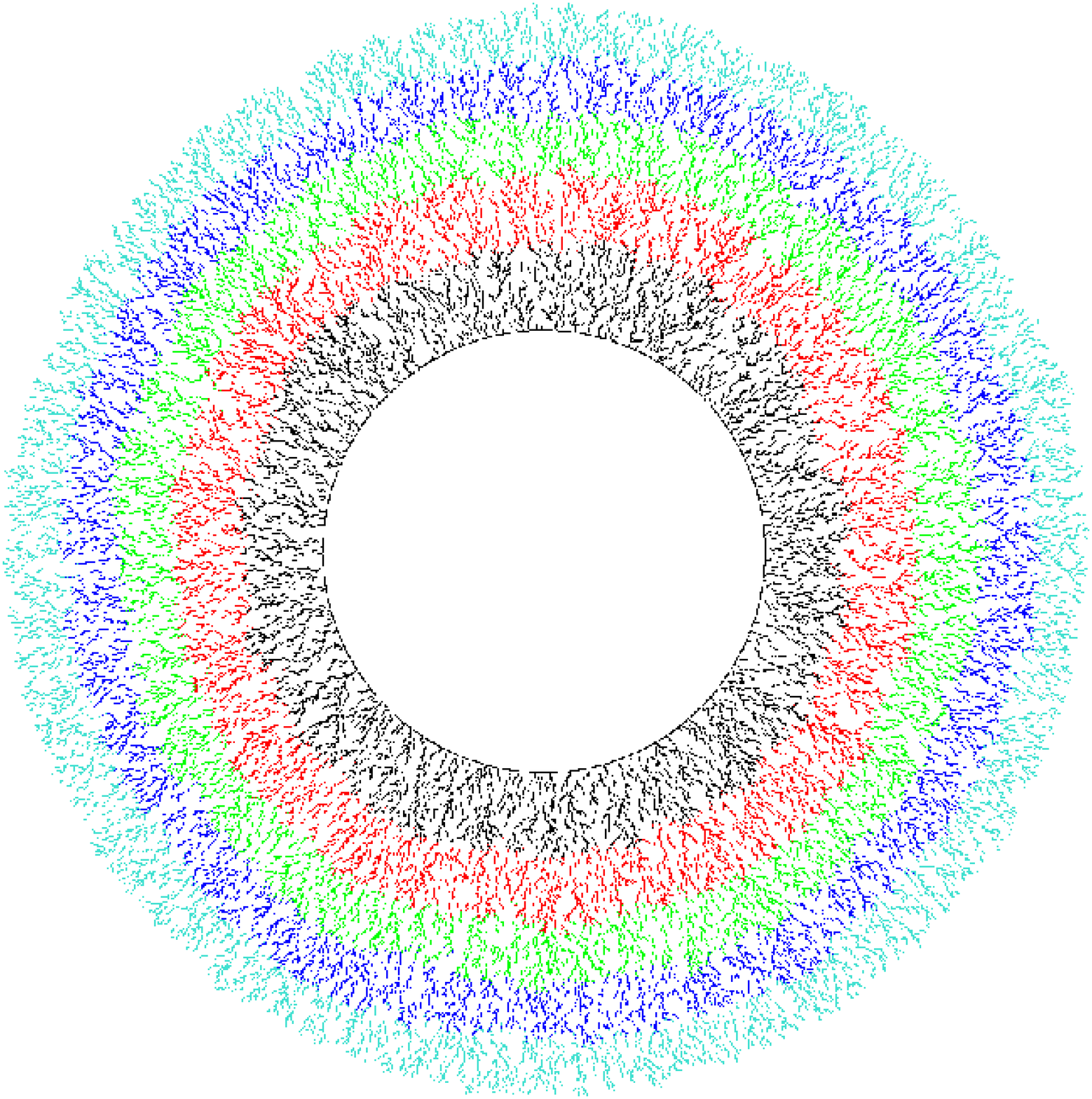}}
		\\
    \subfigure[$\mathrm{HL}(0.5, 0.2)$]
      {\includegraphics[width=0.33 \textwidth]{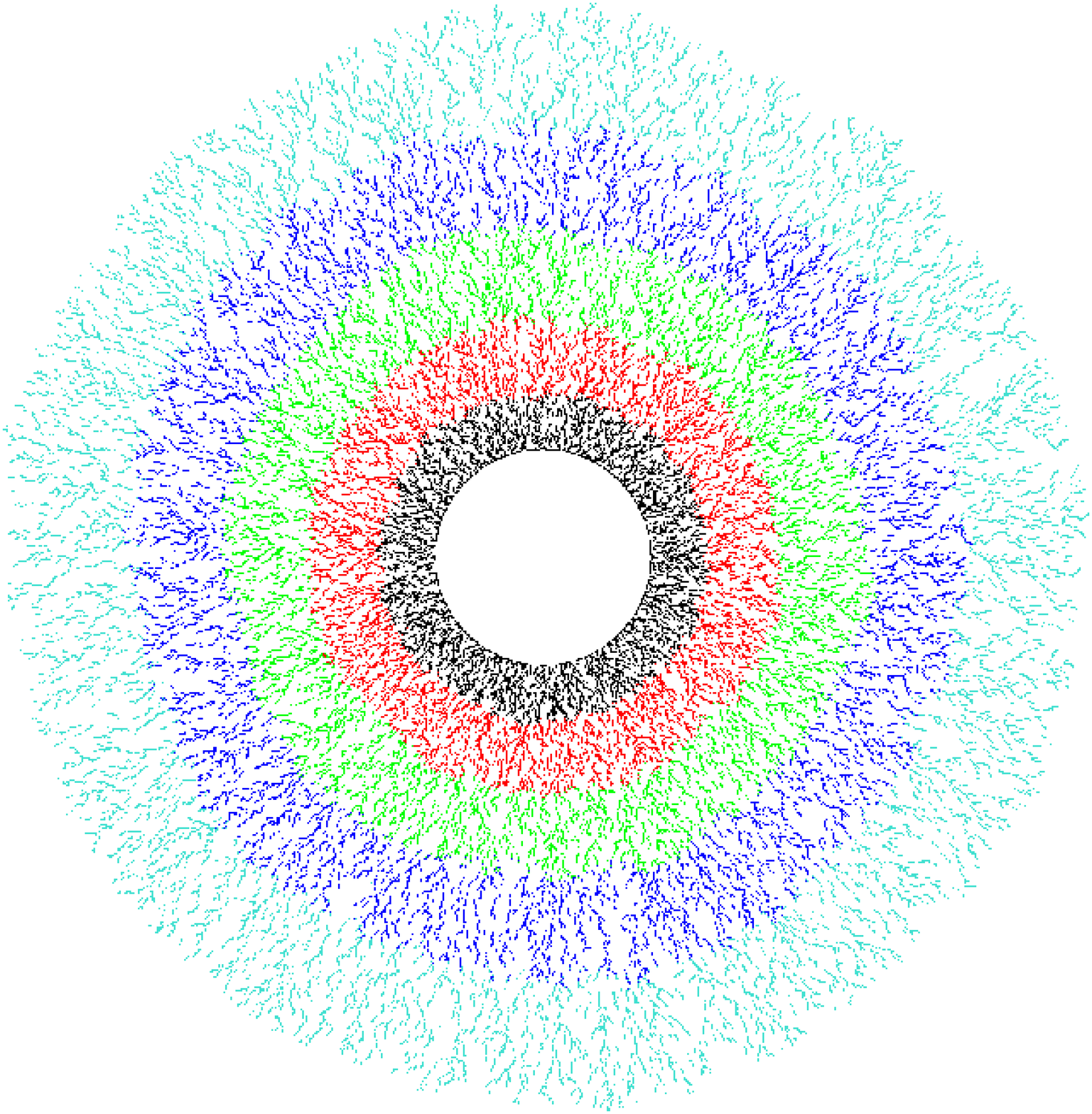}}
    \hfill
    \subfigure[$\mathrm{HL}(2, 0.2)$]
      {\includegraphics[width=0.33 \textwidth]{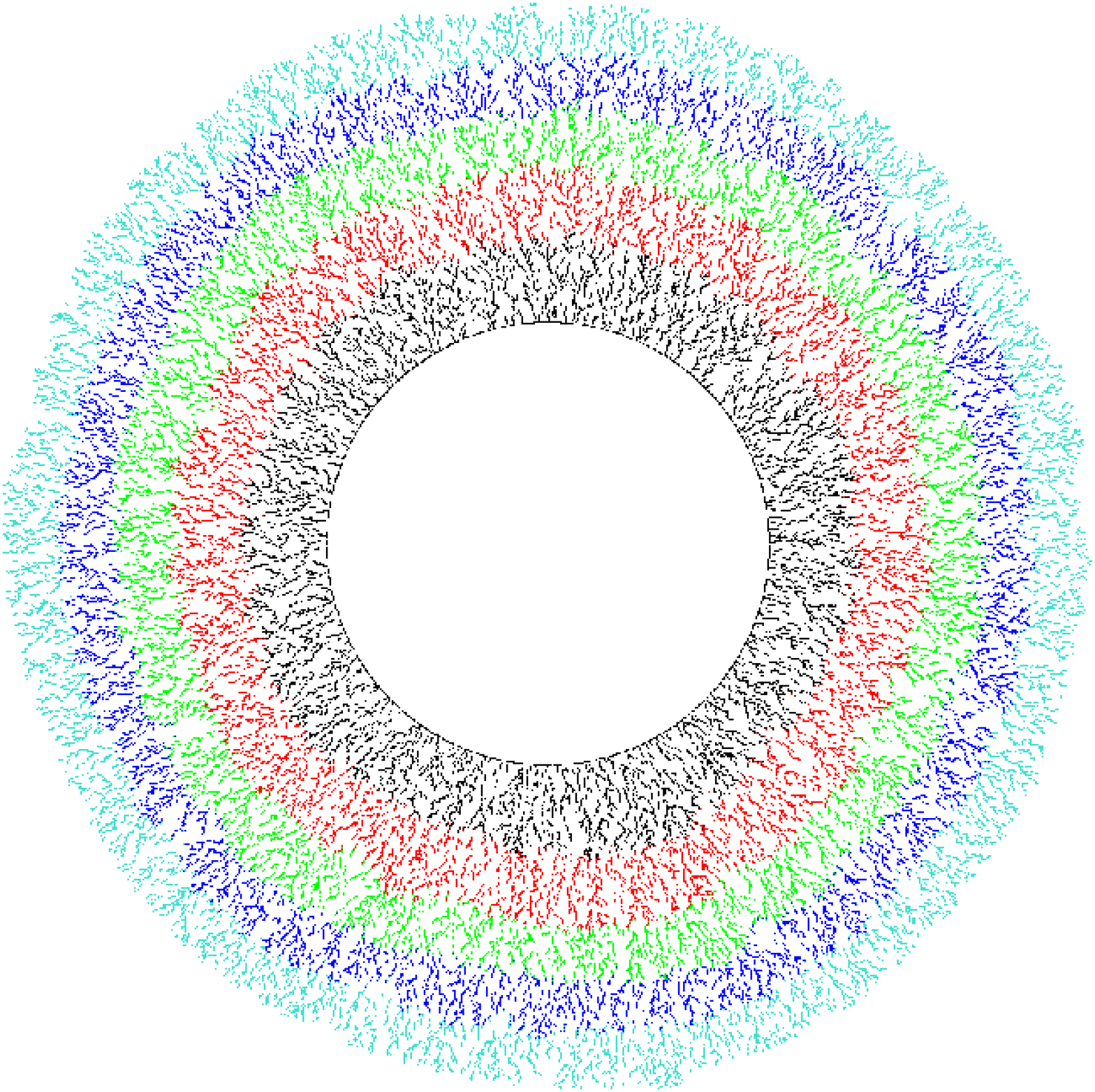}}
    \\
		\subfigure[$\mathrm{HL}(0.5, 0.02)$]
      {\includegraphics[width=0.33 \textwidth]{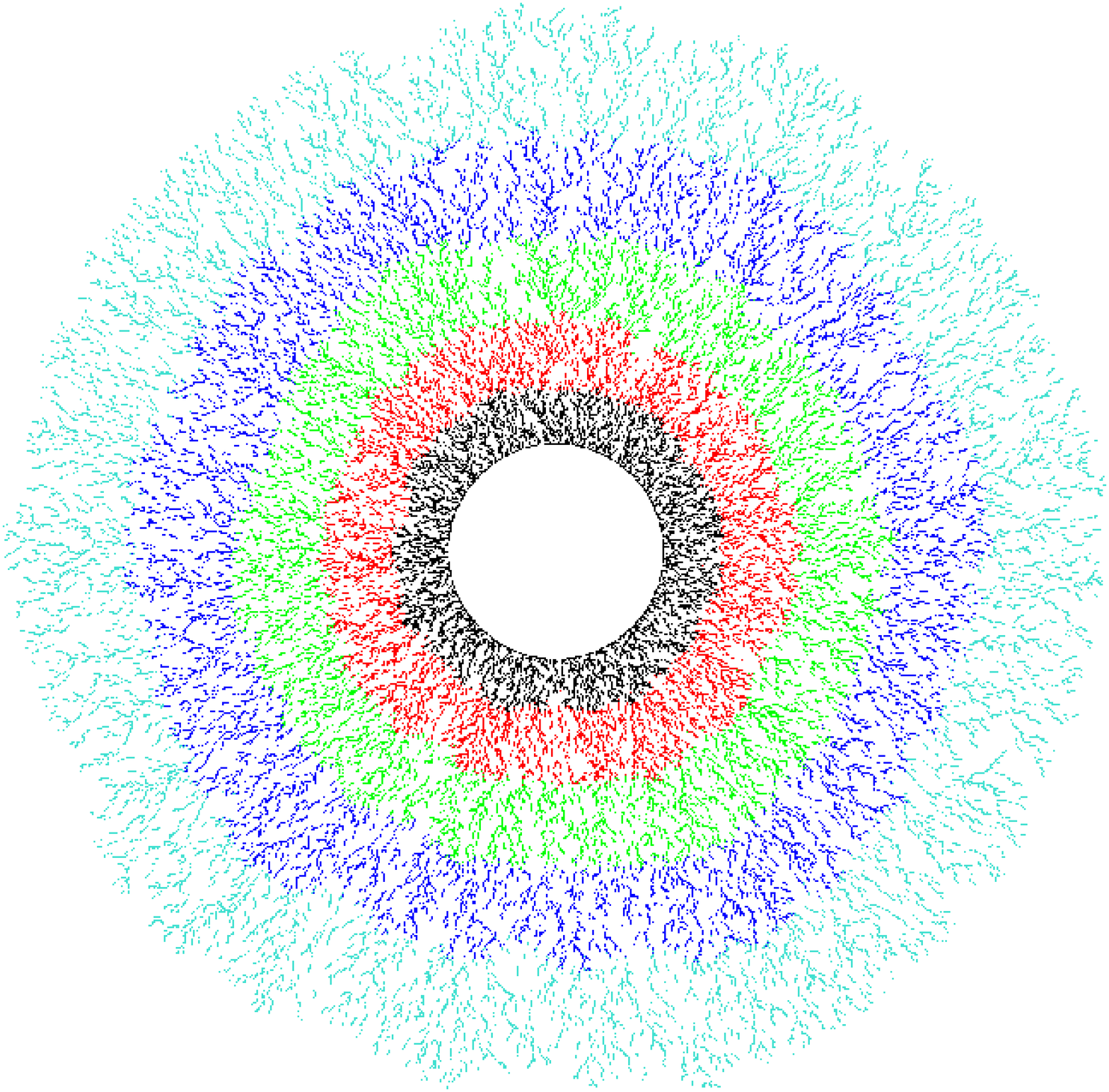}}
    \hfill
    \subfigure[$\mathrm{HL}(2, 0.02)$]
      {\includegraphics[width=0.33 \textwidth]{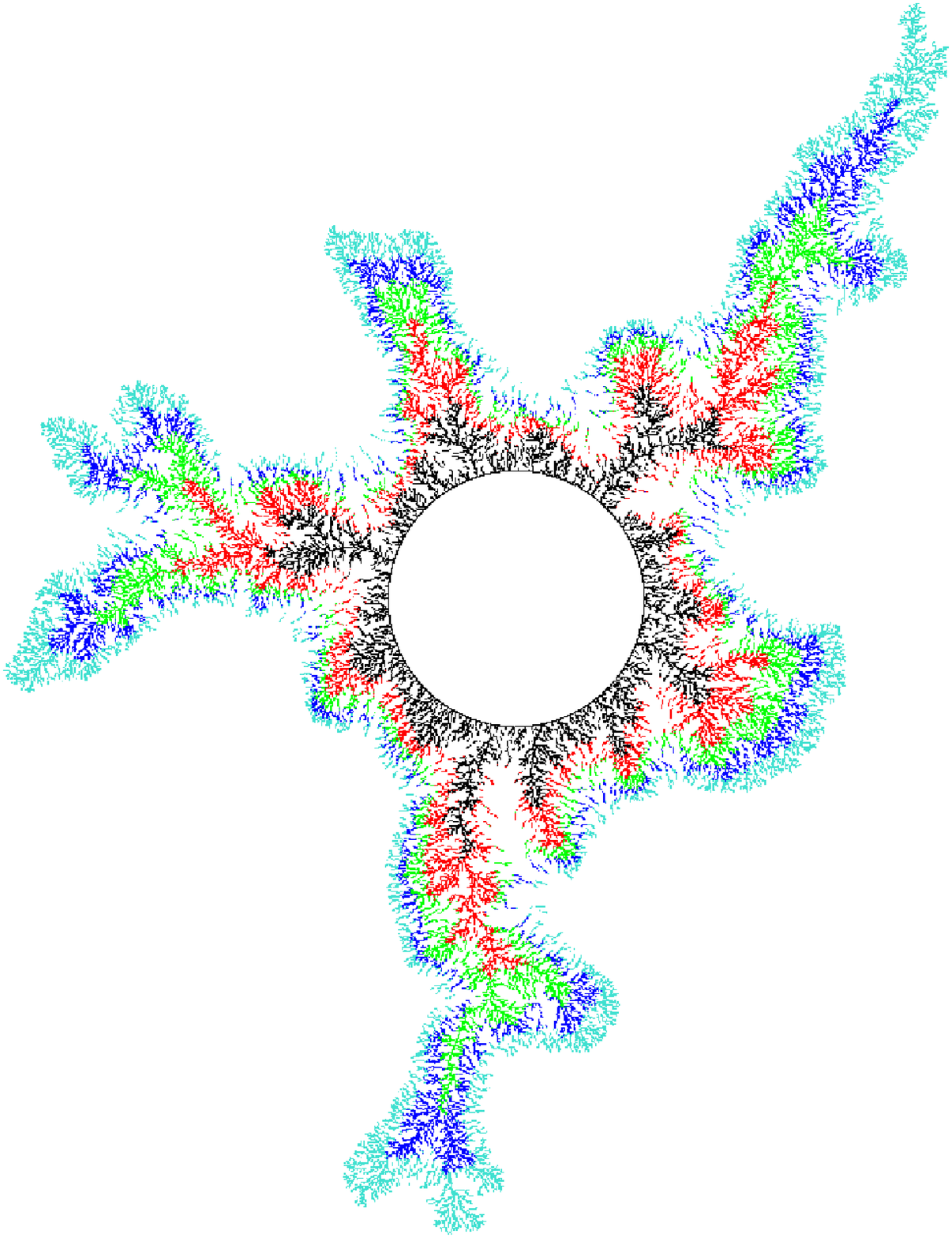}}
  \caption{\textsl{$\mathrm{HL}(\alpha, \parsig)$ clusters with $25,000$ particles when $\capc=10^{-4}$ for $\alpha=0.5$ (left) and $\alpha=2$ (right) when $\parsig=1$, $\parsig=0.2=2(\log \capc^{-1})^{-1/2}$ and $\parsig=0.02=d$. Particles arriving within the same epoch of $5,000$ arrivals have the same colour. }}
  \label{HLpics}
\end{figure}

Figure \ref{HLpics} shows simulations of clusters generated using the same capacity and angle sequences as above. 
When $\parsig$ is sufficiently large, the clusters grow as inflating disks with growth rates as predicted by Theorem \ref{MainThm1}; when $\parsig$ is on the order of the 
basic particle size $d$, the cluster is still an inflating disk for $\alpha =0.5$, whereas for $\alpha=2$ the cluster appears to be growing randomly shaped fingers. In fact, when $\parsig=d$, all simulations of clusters with $\alpha \in (0,1)$ appear to produce disks, whilst those with $\alpha \in (1,2]$ appear to produce random cluster shapes. This is consistent with the claim made by Hastings and Levitov that the cluster growth undergoes a phase transition as $\alpha$ increases through the point $\alpha_{\mathrm{crit}}=1$. We cannot at present prove any statements regarding the existence of limit clusters in this regime, or the possible presence of a sharp phase transition phenomenon.   

\begin{figure}[h!]
    \subfigure[$\mathrm{HL}(0.5, 1)$]
      {\includegraphics[width=0.4 \textwidth]{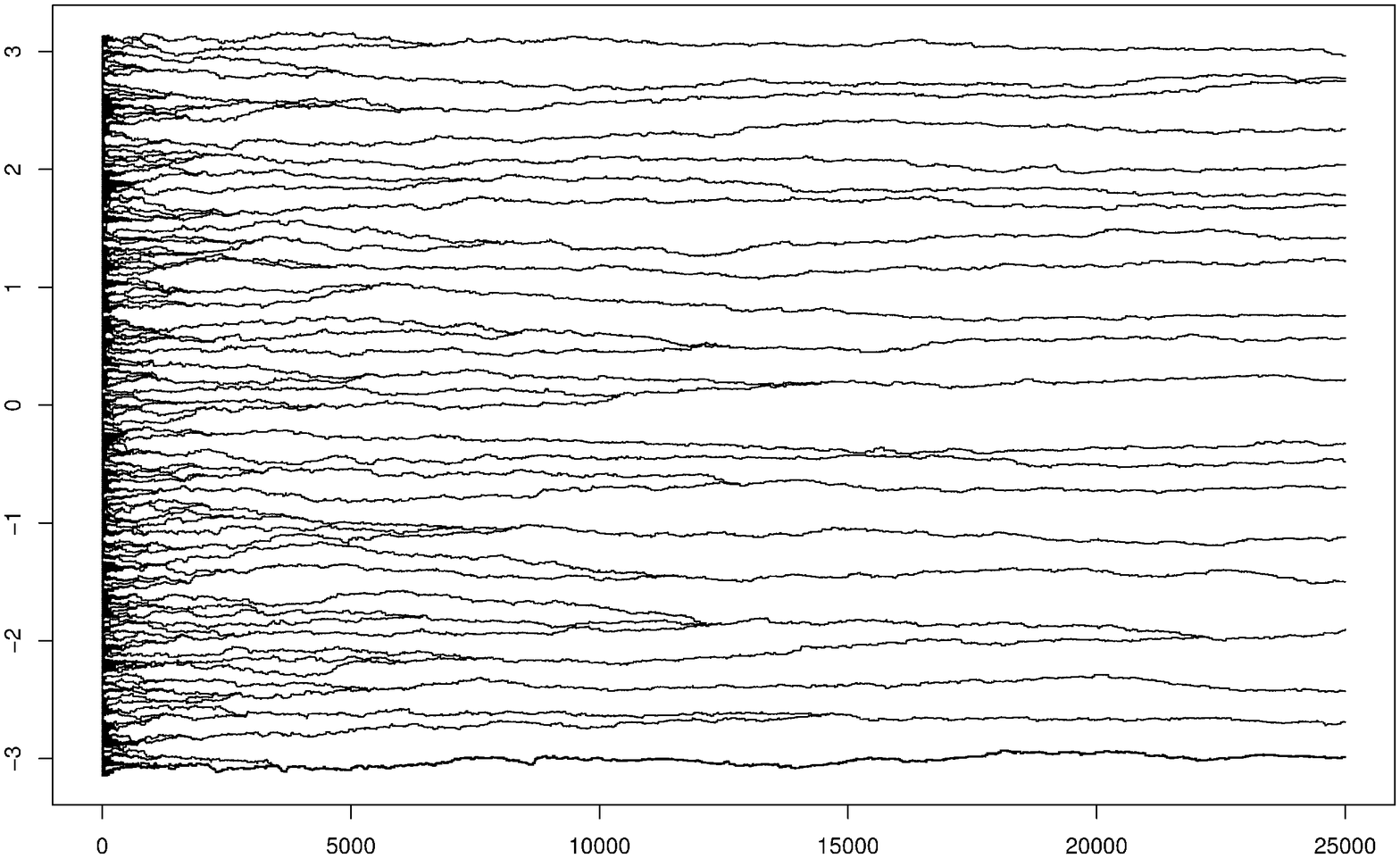}}
    \hfill
    \subfigure[$\mathrm{HL}(2, 1)$]
      {\includegraphics[width=0.4 \textwidth]{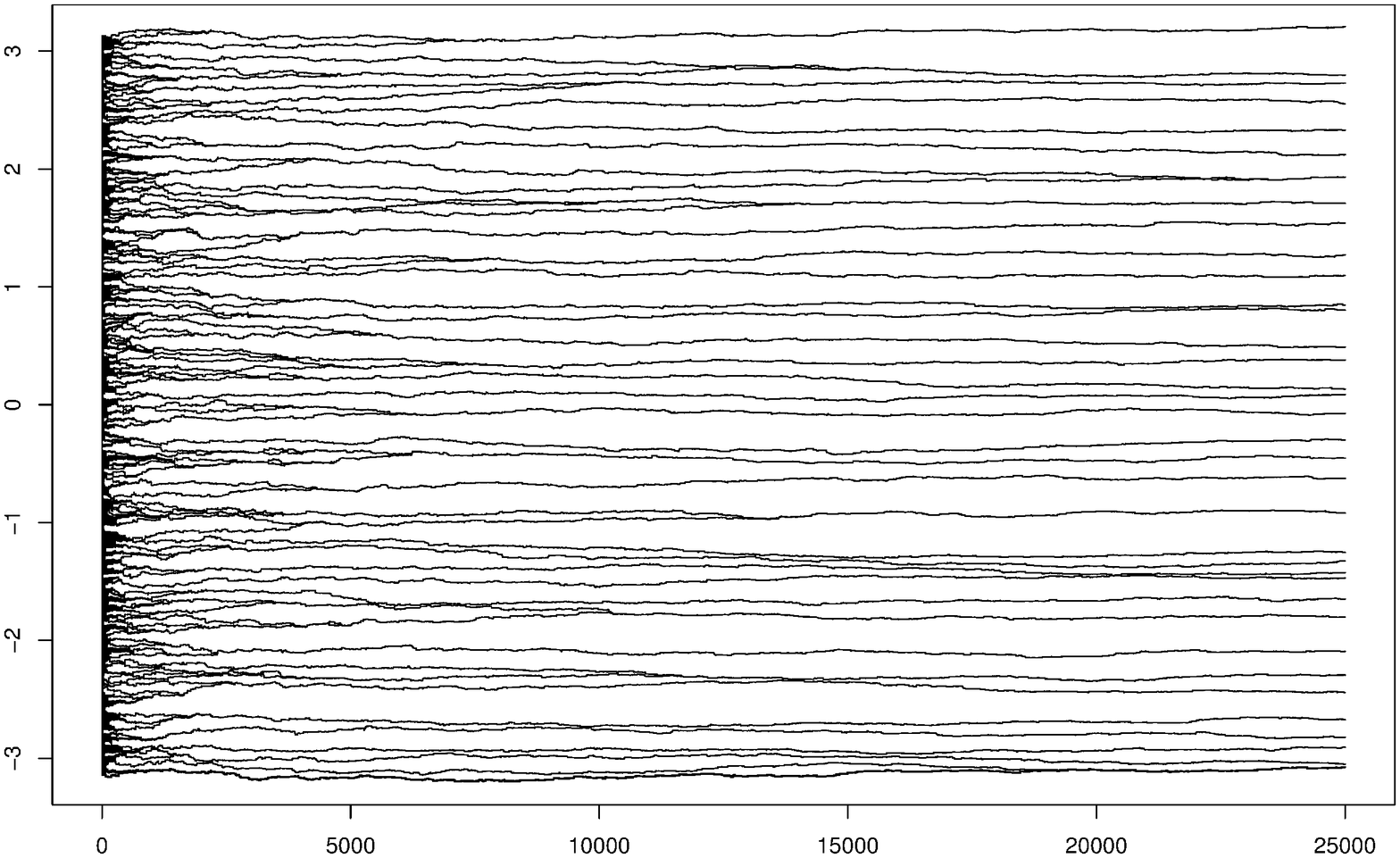}}
		\\
    \subfigure[$\mathrm{HL}(0.5, 0.2)$]
      {\includegraphics[width=0.4 \textwidth]{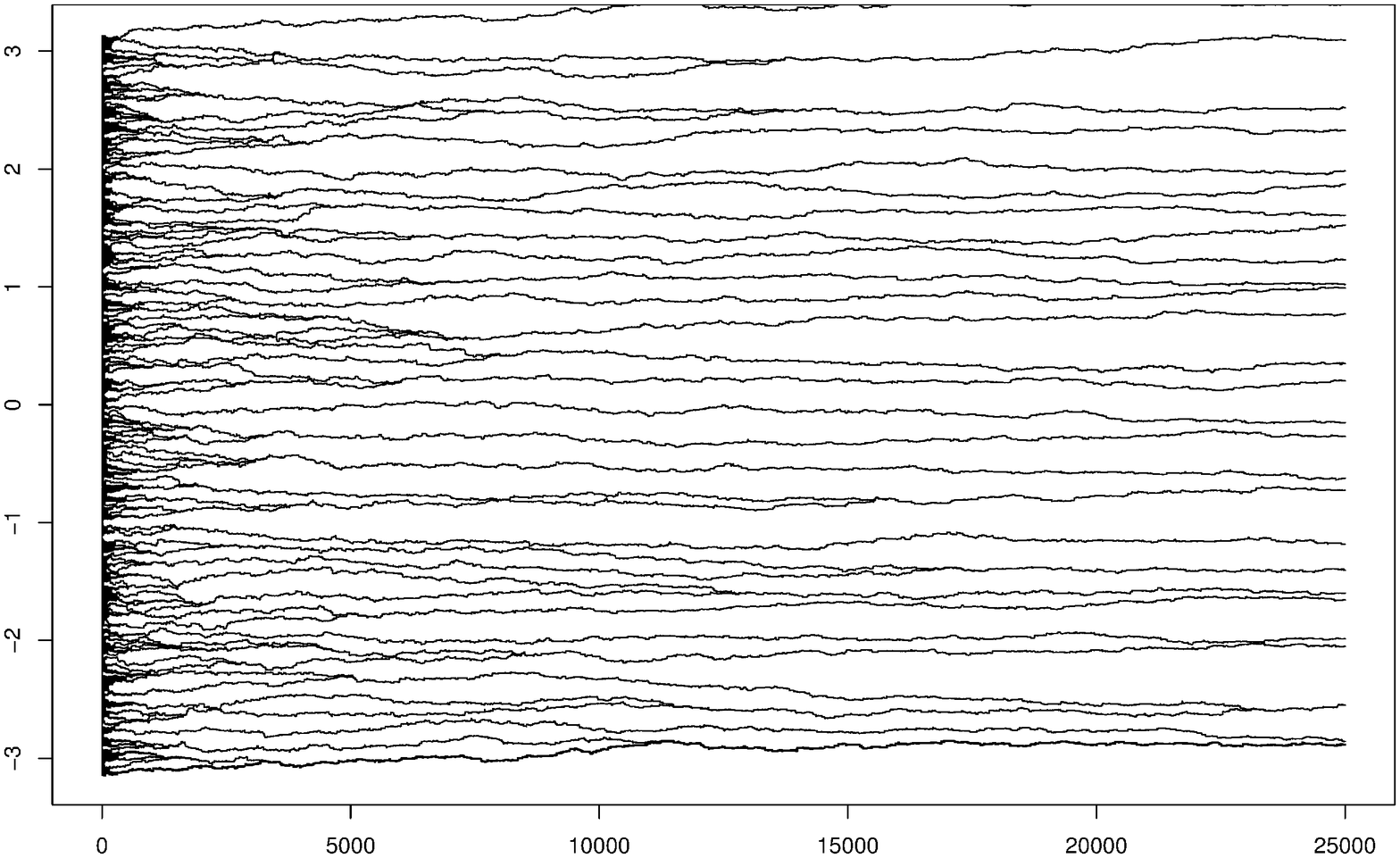}}
    \hfill
    \subfigure[$\mathrm{HL}(2, 0.2)$]
      {\includegraphics[width=0.4 \textwidth]{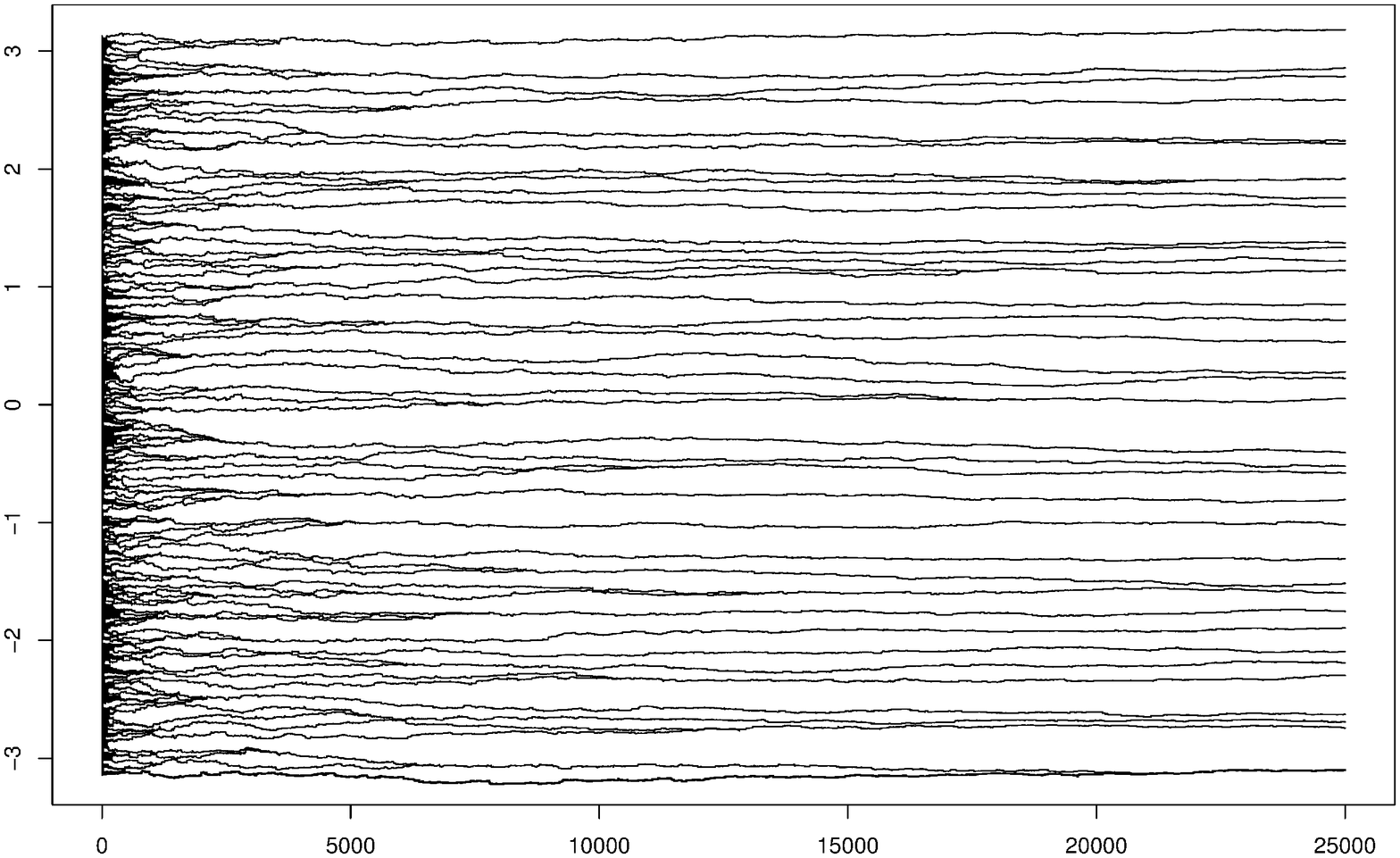}}
    \\
		\subfigure[$\mathrm{HL}(0.5, 0.02)$]
      {\includegraphics[width=0.4 \textwidth]{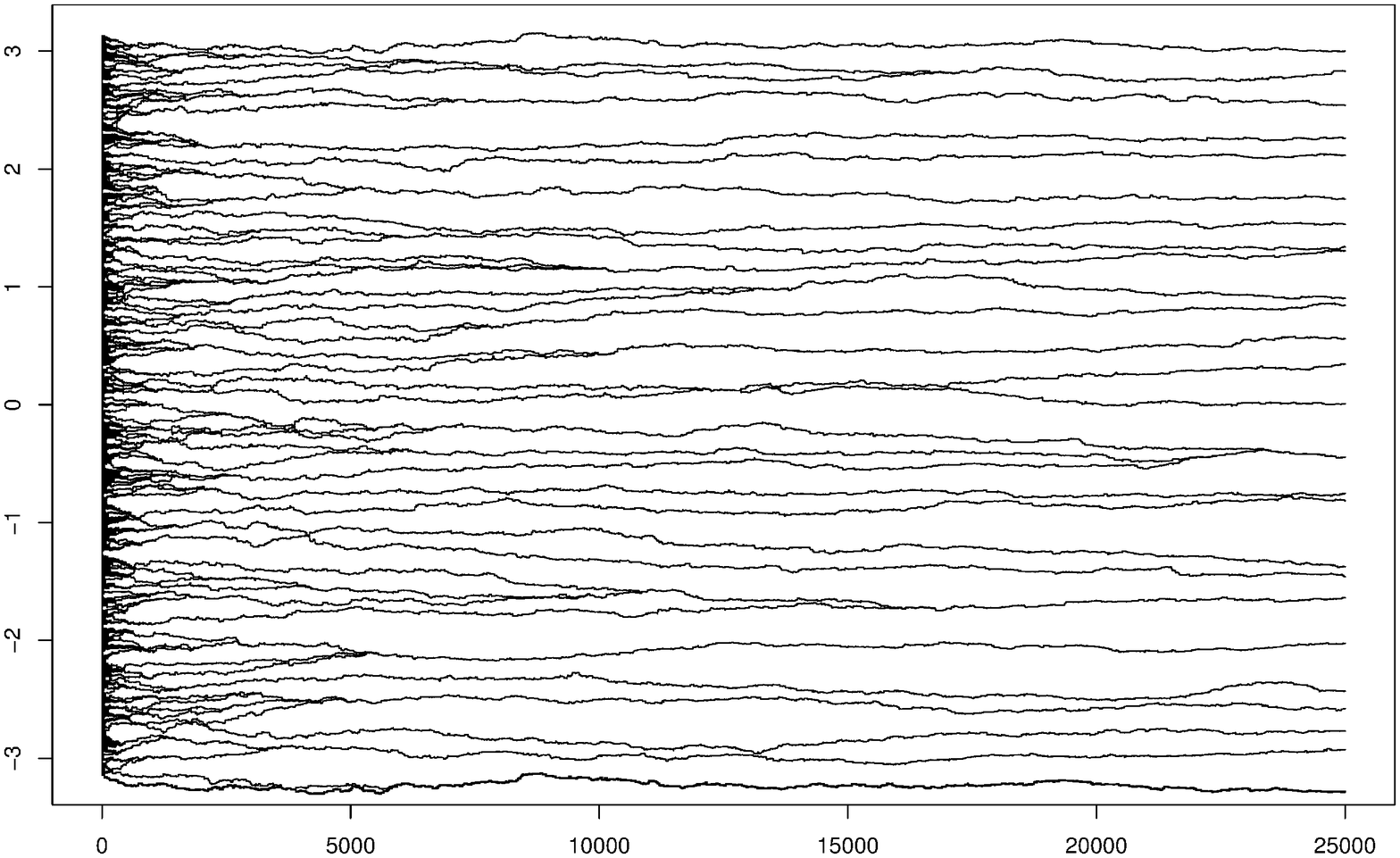}}
    \hfill
    \subfigure[$\mathrm{HL}(2, 0.02)$]
      {\includegraphics[width=0.4 \textwidth]{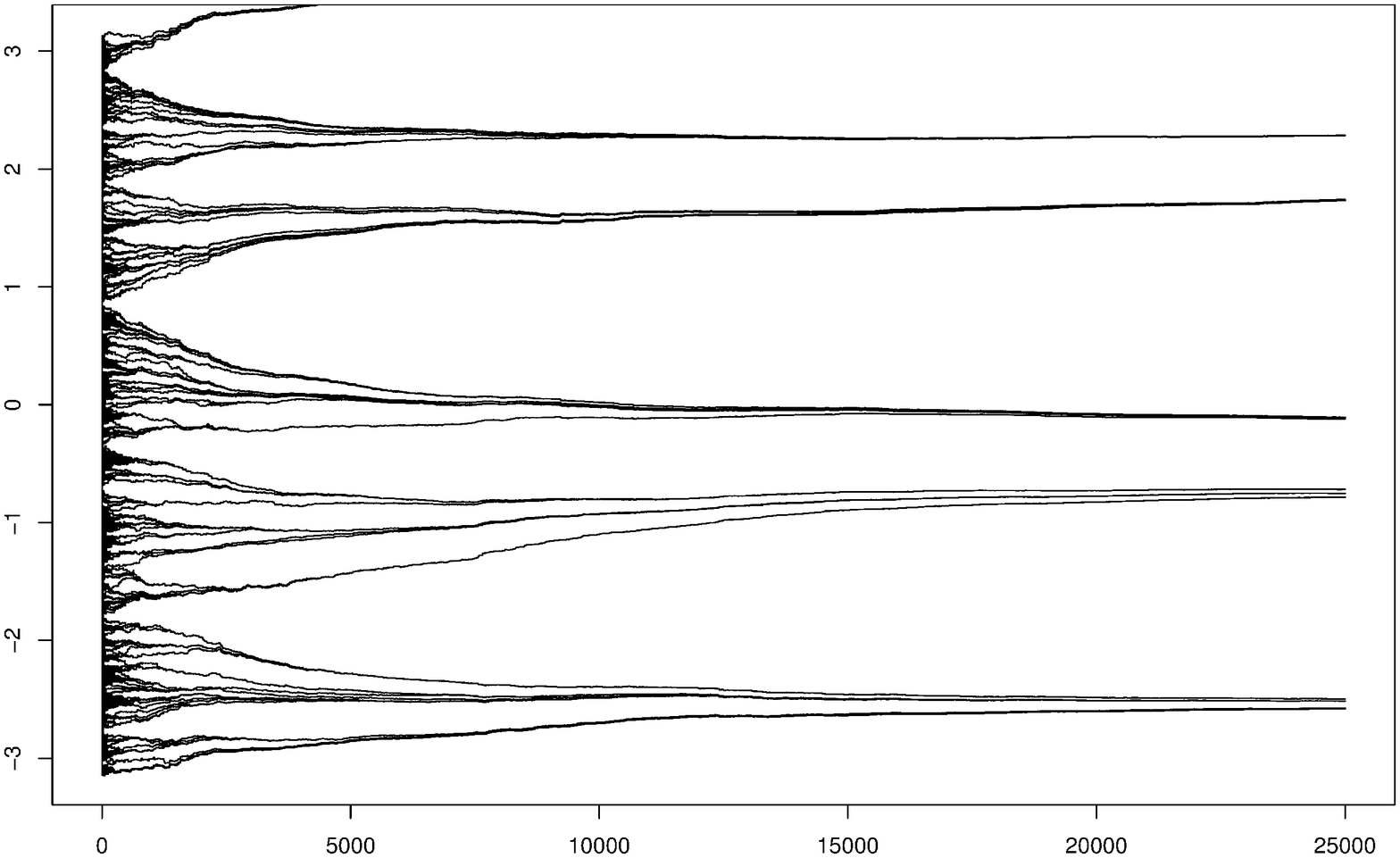}}
  \caption{\textsl{Harmonic measure flow for $\mathrm{HL}(\alpha, \parsig)$ with $25,000$ particles when $\capc=10^{-4}$ for $\alpha=0.5$ (left) and $\alpha=2$ (right) when $\parsig=1$, $\parsig=0.2=2(\log \capc^{-1})^{-1/2}$ and $\parsig=0.02=d$.}}
  \label{flowfig}
\end{figure}

Figure \ref{flowfig} illustrates the evolution of harmonic measure on the cluster boundary under 
the aggregation process. Gaps between flow lines correspond to to the harmonic 
measure carried by the fingers of the cluster that are attached between the corresponding points on the circle. 
When $\parsig$ is sufficiently large, the flows are close to the identity flow as is asserted in Theorem \ref{flowconvergence}; when $\parsig$ is on the order of the 
basic particle size $d$, the flow is still close to the identity flow for $\alpha =0.5$, whereas for $\alpha=2$ it is suggestive of a random, anisotropic scaling limit, and is reminiscent of the general features of the 
(deterministic) anisotropic flows in \cite{JST12}. As time increases, five large gaps appear between the flow lines which corresponds to the five fingers growing in Figure \ref{HLpics}(f). These gaps appear to be asymptotically stable, suggesting that, once established, the fingers will persist indefinitely. 
In fact, when $\parsig=d$, all simulations of flows with $\alpha \in (0,1)$ appear to produce identity flows, whilst those with $\alpha \in (1,2]$ appear to produce random anisotropic limits with asymptotically stable gaps, again suggestive of a phase transition as $\alpha$ increases through the point $\alpha_{\mathrm{crit}}=1$. 

\section{Preliminaries}\label{prelim}

In this section we firstly set out the notation that we will use in the remainder of the paper. We then state some estimates that will feature in our proofs later on.   

\subsection*{Notation}

Throughout this paper, $\{\theta_k\}_{k=1}^{\infty}$ denotes a sequence of independent random variables, uniformly distributed on $[0, 2 \pi)$. Let 
$\{\mathcal{F}_n\}_{n=1}^{\infty}$ be the filtration generated 
by the sequence of angles: $\mathcal{F}_n=\sigma(\theta_1,\ldots, \theta_n)$. With the exception of Section \ref{norristurner}, $\{c_k\}_{k=1}^{\infty}$ will denote the sequence of capacities defined in \eqref{sigregdef}, where $\capc>0$ is an ``initial'' capacity that will tend to zero to enable us to obtain scaling limits. In Section \ref{norristurner}, $\{c_k\}_{k=1}^{\infty}$ will be used to denote any deterministic sequence of capacities with $0<c_k\leq \capc$. The sequence of slit lengths $\{d_k\}_{k=1}^{\infty}$ is calculated from the sequence of capacities using the relation \eqref{capsizerel}.

The building blocks of our clusters are the rotated slit maps $f_k$ defined in \eqref{detmaps}, and these are composed to form the cluster maps $\Phi_n$ defined in \eqref{Phidef}.
More generally, we shall also consider the maps
\begin{equation}
\Phi_{n,m}=f_m \circ\cdots \circ f_n \colon \Delta \rightarrow D_{n,m}, \quad \textrm{for}\quad 1\leq m \leq n.
\label{Phinmdef}
\end{equation}
Note that with this notation, $\Phi_{n,1}=\Phi_n$. 
The map $\Phi_{n,m}$ has capacity
\[C_{n,m} = \sum_{k=m}^n c_k\]
and we write $C_n=C_{n,1}$.

In our analysis, we 
also need to consider the inverse conformal maps 
\[g_{\capc}=f_{\capc}^{-1} \colon \Delta \setminus (1, 1+d] \rightarrow \Delta \]
 and the corresponding $g_k = f_k^{-1}$, as well as
\[\Gamma_{n}=(\Phi_{n})^{-1}\colon D_n\rightarrow \Delta, \] 
along with the functions $\Gamma_{n,m}=(\Phi_{n,m})^{-1}$.

We often find it convenient to use logarithmic coordinates. To 
this end, 
we write 
\[\CC_{\parsig+}=\{z\in \CC: \mathrm{Re}(z)>\parsig \} \quad \textrm{and} \quad 
\CC_+=\CC_{0+},\]
define $\tilde{g}_{\capc}(z)=\log(g_{\capc}(e^z))$,
and setting  $\tilde{D}_{n,m} = \{z: e^z \in D_{n,m} \}$ we introduce
\[\tilde{\Gamma}_{n,m}(z)=\log\left[\Gamma_{n,m}(e^z)\right], 
\quad z \in \tilde{D}_{n,m}. \]
The map $\tilde{g}_{\capc}$ extends locally in a continuous way to the imaginary axis so we can define a continuous map $\g_\capc$ from the open set $\mathbb{R} \setminus 2 \pi \mathbb{Z}$ to itself by
\[
\gamma_\capc(x) = \IIm \tilde{g}_{\capc} (ix), \quad x\in(0,2 \pi).
\]

For each of the above definitions we also define ``starred'' versions in which the capacity sequence $\{c_k\}_{k=1}^\infty$ is replaced by the sequence $\{c_k^*\}_{k=1}^\infty$ defined in \eqref{detcaps}. We denote these using the notation $f_k^*, \Phi_n^*$, and so on.

\subsection*{Estimates}

Our analysis relies heavily on bounds on the building block 
$f_{\capc}(z)$ and its derivate. More 
precisely, in our coupling argument, 
we need to be able to compare maps associated 
with different capacity increments, evaluated at the same point.
We list below a collection of such estimates. We use the letter $A$ to indicate generic absolute constants. In cases where it is useful to track these absolute constants, we add subscripts. 
%Where we have bounds that depend on parameters such as $\parsig$ that we do not need to track explicitly we shall sometimes use notation such as $A(\parsig)$. 

When working with slit maps, estimates 
can obtained by performing an asymptotic analysis of the explicit expressions 
for the mapping: one finds that if $|z-1|>2 d$, then, as 
$\capc \rightarrow 0$,
\begin{equation*}
f_{\capc}(z)=z\left(1+\capc\frac{z+1}{z-1}\right)
+\mathcal{O}\left(\frac{\capc^2}{(z-1)^3}\right),
\label{mapasymp}
\end{equation*}
\begin{equation*}
f'_{\capc}(z)=1+\capc\left(1-\frac{2}{(z-1)^2}\right)
+\mathcal{O}\left(\frac{\capc^2}{(z-1)^3}\right),
\label{realderasymp}
\end{equation*}
and 
\begin{equation*}
\frac{1}{f'_{\capc}(z)}=1+\capc\left(\frac{2}{(z-1)^2}-1\right)
+\mathcal{O}\left(\frac{\capc^2}{(z-1)^3}\right).
\label{derasymp}
\end{equation*}
However, using standard results from Loewner theory yield 
cleaner bounds that are easier to extend to more general particles. Let $\{p_t(z)\}_{t \ge 0}$ be a family of analytic functions such that for each $t$, $p_t(z)$ has strictly positive real part for $z \in \Delta$. The 
Loewner differential equation,
\begin{equation}
\dot{\varphi}_t(z)=z\varphi'_t(z)p_t(z), \quad z\in \Delta, t>0,
\label{loewnereq}
\end{equation}
together with the initial condition $\varphi_0=z$ then parametrizes a family $
\{\varphi_t\}_{t\geq 0}$ of 
conformal maps  
\[\varphi_t\colon \Delta \rightarrow D_t,\]
onto simply connected domains $D_t$ forming a decreasing sequence, 
with expansions at infinity of the form $\varphi_t(z)=e^tz+\mathcal{O}(1)$. 
%The 
%functions $\{p_t\}_{t>0}$ in \eqref{loewnereq} are analytic and satisfy
%\[\mathrm{Re}(p_t(z))>0, \quad z \in \Delta,\]
Notice that there is a family of probability measures $\{\mu_t\}_{t>0}$ on the unit circle such that
\[p_t(z)=\int_{\TT}\frac{z+\zeta}{z-\zeta}d\mu_t(\zeta).\]
In this parametrization of the Loewner mappings, the time parameter 
$t>0$ corresponds precisely to the total capacity of 
cluster $K_t=\CC\setminus D_t$ via $\mathrm{cap}(K_t)=e^t$.   

In particular, the mappings onto the complement of $\mathrm{HL}(\alpha, \parsig)$ clusters 
can be obtained by considering the measure-valued process with
\begin{equation}
d\mu_t(e^{i\theta})=\sum_{k=1}^{\infty}\mathbf{1}_{[C_{k-1},C_k)}(t)\delta_{\theta_k},
\label{drivingmeasures}
\end{equation}
solving the corresponding Loewner equation, and setting 
\[\Phi_n=\varphi_{C_n}.\]
We refer the reader to \cite{CM01, Lawbook, JST12} for background material on the 
Loewner equation in the present context.

The following result is stated in \cite{F1}  
in a more general setting (see \cite[Section 3]{FL11} for a proof in the chordal case; the radial case is entirely similar). We have adapted the 
notation to match that of the present paper. 
In our case, the Loewner driving terms (in the notation of \cite{F1}) 
are of the form 
$W(t)=\Theta \mathbf{1}_{[0,\capc)}(t)$, where $\Theta$ has 
uniform distribution on $[0,2\pi]$. 

\begin{lemma}\label{mapestimate}
There exists a 
constant $0<A_0<\infty$ such that, if $|z|-1>2(d\vee d^*)$, then
\begin{equation}
\label{fratio}
e^{-A_0\frac{|\capc-\capc^*|}{(|z|-1)^2}}\leq 
\left|\frac{f'_{\capc}(z)}{f'_{\capc^*}(z)}\right|\leq e^{A_0\frac{|\capc-\capc^*|}{(|z|-1)^2}}
\end{equation}
and
\begin{equation}
|f_{\capc}(z)-f_{\capc^*}(z)|\leq A_0|f'_{\capc \wedge \capc^*}(z)|(|z|-1)
\left(e^{A_0\frac{|\capc-\capc^*|}{(|z|-1)^2}}-1\right).
\label{loewnerfunction}
\end{equation}
% \br Do we need $e^{\parsig}-1>2(d(\capc)\vee d(\capc^*))$ here? If, so must put this condition in in subsequent uses of this lemma. \er
\end{lemma}
We shall implement this result in the following guise. % \br Maybe don't need this. \er
\begin{lemma}\label{capreplacement}
There exists a 
constant $0<A_0<\infty$ such that, if $|z|-1>2(d\vee d^*)$, then
\begin{equation}
\left|\log\left|\frac{f'_{\capc}(z)}{f'_{\capc^*}(z)}\right|\right|\leq A_0\frac{|\capc-\capc^*|}{(|z|-1)^2}.
\label{capderivative}
\end{equation}
Further, there exists a constant $1<A_1<\infty$ such that if $(|z|-1)^2 > A_1 (\capc \vee \capc^*)$ 
%\bg (can replace 2 by any constant $> 1$ if necessary. These two conditions are asymptotically essentially %equivalent) \eg. 
then %there exists some constant $A_1>1$ such that
\begin{equation}
|f_{\capc}(z)-f_{\capc^*}(z)|\leq A_1\frac{|\capc-\capc^*|}{|z|-1}.
\label{capfunction}
\end{equation}
% \begin{equation}
% \left|\left|\frac{f'_{\capc}(z)}{f'_{\capc^*}(z)}\right|^{\alpha}-1\right|\leq C_0\alpha\frac{|\capc-\capc^*|}{(e^{\parsig}-1)^2},
% \label{capcderivative}
% \end{equation}
% and
% \begin{equation}
% |f_{\capc}(z)-f_{\capc^*}(z)|\leq K_1\frac{|\capc-\capc^*|}{e^{\parsig}-1}|f'_{\capc\wedge \capc^*}(z)|\leq c_1(\parsig)\frac{|\capc-\capc^*|}{(e^{\parsig}-1)^2}.
% \label{capcfunction}
% \end{equation}
\end{lemma}
\begin{proof}
The first statement follows directly from the previous lemma.
The inequality \eqref{capfunction} can be deduced from \eqref{loewnerfunction} 
once we have established that $|f'_{\capc\wedge \capc^*}(z)|\leq A$ for some absolute constant $A<\infty$. But this is immediate, by setting $\capc^*=0$
in \eqref{fratio} and using our conditions on $|z|$.
%To do this, 
%we assume without loss of generality that $1>\capc>\capc^*$ and 
%note that the condition $e^{\parsig}>1+2d$ implies that $g'_{\capc}$ is 
%a non-vanishing analytic function in the domain 
%$e^{\parsig}\Delta \supset f_{\capc}(e^{\parsig}\Delta)$. Hence, by the minimum 
%modulus principle, 
%\[\min_{w\in f_{\capc}(e^{\parsig}\Delta)}|g'_{\capc}(w)|\geq 
%\min_{w\in e^{\parsig}\Delta}|g'_{\capc}(w)|.\] 
%The existence of the required constant $A<\infty$ now follows from 
%the fact that $|f'_{\capc}(z)|=1/|g'_{\capc}(f_{\capc}(z))|$, the lower bound 
%on $g'_{\capc}$ contained in the estimates in \cite[Eq. (4)]{NT12}, and our 
%assumptions on $\parsig$ and $\capc$.
%% In the case of the slit map, the assertion follows from an 
%% explicit computation.
\end{proof}
% By distortion estimates, we have 
% \[|f'_{\capc\wedge \capc^*}(e^{\parsig+i\theta})|\leq \frac{e^{2\parsig}}{(e^{\parsig}+1)(e^{\parsig}-1)},\] 
% which is the 
% reason for the last inequality in \eqref{capcfunction}.\footnote{This is wasteful, as we could compute with the slit map as in \eqref{realderasymp} to push down the power to $(e^{\parsig}-1)^{-1}$. On the other hand, we already have a squared power of the distance in the first inequality (and that's a double-sided bound so we can't do much better), and the results are particle-independent if we use Loewner bounds.}

We recall some of the facts about the inverse maps $g_{\capc}=f_{\capc}^{-1}$ in logarithmic coordinates proved in \cite{NT12} (see the beginning of this section for definitions of $\tilde{g}$ and $\gamma$).

\begin{lemma}\label{gintegral}
If $\mathrm{Re}(z)>d$, then
\begin{equation}
\frac{1}{2\pi}\int_0^{2\pi}\left(\tilde{g}_{\capc}(z-i\theta)-(z-i\theta)\right)d\theta=-\capc.
\end{equation}
\end{lemma}

Moreover, for any symmetric particle, 
$$
\int_0^{2\pi} \tg_\capc(x) dx = 0,
$$
where $\tg_{\capc}(x) = \g_{\capc}(x)-x$.

In the case of slit maps, we can explicitly calculate $\g_{\capc}$ as
\[
 \g_{\capc}(x) = 2 {\rm sgn}(x) \tan^{-1} \sqrt{e^\capc \tan^2(x/2) + e^\capc - 1 }, \quad x \in [-\pi/2, \pi/2] \setminus \{0\}.
\]
We use this in the following lemma.

\begin{lemma}\label{gmapestimate}
 For all $x \in \mathbb{R}\setminus 2 \pi \mathbb{Z}$,
 \[ 
  |\g_{\capc}(x) - \g_{\capc^*}(x)| \leq \frac{2|\capc - \capc^*|}{|\tan(x/2)| \vee \sqrt{\capc \wedge \capc^*}}. 
 \]
\end{lemma}
\begin{proof}
 Let $\dot{\g}_\capc(x)$ denote differentiation with respect to $\capc$. Then, for $x \notin \mathbb{Z}$,
 \begin{align*}
  |\dot{\g}_\capc(x)| &= \frac{2}{\sqrt{e^\capc \tan^2(x/2) + e^\capc - 1}} \\
  &\leq \frac{2}{|\tan(x/2)| \vee \sqrt{\capc}}.
 \end{align*}
 The result follows from the Mean Value Theorem.
\end{proof}

Define the function $\rho(\capc)$ by
\begin{equation}
\label{rhodef}
1 = \frac{1}{2 \pi}\rho(\capc)\int_0^{2 \pi}  \tg_{\capc}(x)^2 dx.
\end{equation}
The following result is shown in \cite{NT12}.
\begin{lemma}\label{gmapestimatesnt}
 There is an absolute constant $A_2$ such that for $0<|x|<\pi$,
 \[
  |\tg_{\capc}(x)| \leq \frac{A_2 \capc}{|x| \vee \sqrt{\capc}}.
 \]
 Moreover, $A_2$ may be chosen so that
 \[
  \capc^{3/2}/A_2 \leq \rho(\capc)^{-1} \leq A_2 \capc^{3/2}, 
 \] 
 and
 \[
 \frac{1}{2\pi} \int_0^{2 \pi} |\tg_\capc(x) \tg_\capc(x+h)| dx \leq \frac{A_2 \capc^2}{h} \log \left ( \frac{1}{\capc}\right )
 \]
 whenever $h \in [d, \pi]$.
\end{lemma}
Furthermore, in the case of slit maps, it can be shown explicitly that 
\begin{equation}
\label{slitrho}
 \capc^{-3/2} \rho(\capc)^{-1} \to \frac{16}{3 \pi}
\end{equation}
as $\capc \to 0$.

\section{Deterministic capacity sequences}\label{norristurner}

In this section only, we assume that $\{c_k\}_{k=1}^\infty$ is any {\em deterministic} sequence of numbers depending on a parameter $\capc$ with the property that $0< c_k \leq \capc$ for all $k$. In particular, note that the sequence $\{c_k^*\}_{k=1}^\infty$ in \eqref{detcaps} satisfies this condition.\footnote{The results below can be generalized in a straightfoward way to the case where the $c_k$ are random but independent of the sequence of angles $\{\theta_k\}_{k=1}^\infty$, however they can not be extended to the sequence $\{c_k\}_{k=1}^\infty$ defined in \eqref{sigregdef} due to the  dependence of $c_k$ on $\theta_1, \ldots, \theta_{k}$.} The result and argument below is a refinement of \cite[Proposition 5.1]{NT12}
and shows that, with high probability, the conformal map $\Phi_{n}$ is 
close to $z\mapsto e^{C_{n}}z$.

\begin{theorem}
\label{NTmain}
Let $\Phi_n$ denote the conformal map corresponding to particles with deterministic capacities $0<c_k\leq \capc$ and independent uniform angles $0 \leq \theta_k < 2 \pi$. For any $N\in \mathbb{N}$, $\parsig>0$ and $0<\epsilon<\parsig/3$ satisfying 
$\parsig-3\epsilon>d\vee \capc^{1/3}$, we have
\[
\PP\left(\sup_{z\in \CC_{\parsig+}, n\leq N} 
|\tilde{\Phi}_{n}(z)-z-C_{n}|<\epsilon\right)
\geq 1-AN\frac{\capc^{8/5}}{\epsilon^{11/5}[(\parsig-3\epsilon)^{9/5}\wedge 1]}\exp\left(-A(\parsig, \epsilon, N) \frac{\epsilon^2}{\capc}\right),
\]
where 
\begin{equation}
A(\parsig, \epsilon, N)=\frac{A}{\log \frac{1}{\parsig-3\epsilon}+\log(\parsig+N\epsilon)+\frac{\epsilon}{\parsig-\epsilon}}.
\end{equation}
\end{theorem}

Note that this probability tends to 1 provided the argument in 
the exponent tends to $\infty$, and that this can be ensured through a 
judicious choice of parameters $\parsig$ and $\epsilon$.

\begin{proof}
The required bound is established through a series of lemmas that will be stated after the present proof. The proof strategy is to first show that, for fixed $z$, $|\tG_n(z)-z+C_n|$ is small by using exponential martingale estimates (Lemma \ref{expmart}). We then show that this result holds uniformly for all $z$ bounded away from the imaginary axis (Lemma \ref{linecontrol}). The result then follows by the observation that if $w=\tG_n(z)$, then
\[|\tilde{\Phi}_n(w)-w-C_n|=|z-\tG_n(z)-C_n|.\]

Let $\parsig, \epsilon > 0$ be given, satisfying $\parsig-3\epsilon>d$ and fix $N\in \mathbb{N}$.
Set
\[M=M(\capc, N, \epsilon)=\left\lfloor\frac{1}{\epsilon}C_N\right\rfloor,\]
where $C_N=\sum_{j=1}^Nc_{j}$.
For $1\leq m\leq M$, we now define
\[R_m=R_m(\parsig,\epsilon)=\parsig+(m-2)\epsilon.\]
Note that 
\[R_1=\parsig-\epsilon \quad \textrm{while}\quad  R_M\leq \parsig+C_N-2\epsilon.\] Finally, we let $n_m\geq 1$ denote the
largest integer such that \[C_{n_m} \leq R_m-\parsig+2\epsilon.\] 

We now consider points on the line
\[z\in \ell_{R_{m}}=\{w \in \CC \colon \textrm{Re}(w)=R_m\},\] 
and define the following stopping time:
\[
T_{R_m}=\inf\left\{n\geq 1\colon z\notin \tilde{D}_n \right. \left. \,
\textrm{or}\,
\textrm{Re}(\tG_n(z))\leq R_m-C_n-\epsilon \, \textrm{for some}\ z\in \ell_{R_m}\right\}\land n_m.
\]
Because of the normalization of $\Phi_n$ at $\infty$, the function 
$z \mapsto \tG_n(z)-z+C_n$ is a bounded holomorphic function in the 
half-planes 
\[\CC_{R_m+}=\{z \in \CC: \textrm{Re}(z)>R_m\},\] 
and hence, by the maximum principle,
\[\sup_{\CC_{R_m+}}|\tG_n(z)-z+C_n|=\sup_{\ell_{R_m}}|\tG_n(z)-z+C_n|.\]

With this in mind, we define, for each $R_m$, the event
\begin{equation}
\Omega_{R_m}=\left\{\sup_{z \in \ell_{R_m}, n\leq T_{R_m}}|\tG_n(z)-z+C_n|<\epsilon\right\}.
\end{equation}
We then consider the desirable event
\[\Omega_0(N,\epsilon,\parsig)=\bigcap_{m=1}^{M}\Omega_{R_{m}}.\]

For $n\leq N$, we can find a $k\leq M$ such that
\[R_k+\epsilon-\parsig\leq C_n\leq R_k+2\epsilon-\parsig.\]
Thus, if $\textrm{Re}(z)\geq C_n+\parsig-\epsilon$, 
then $\textrm{Re}(z)\geq R_k$. Hence, on the event 
$\Omega_0(N,\epsilon,\parsig)$, it follows that $z \in \tilde{D}_n$ and
$|\tG_n(z)-z+C_n|<\epsilon$.

We note that, on $\Omega_0(N,\epsilon, \parsig)$, we have
\[\textrm{Re}(\tG_n(z))<R_{k}-C_n+\epsilon\leq \parsig.\]
This means that 
for $w \in \CC$ with $\textrm{Re}(w)\geq \parsig$, there exists a
$z_0$ with $\textrm{Re}(z_0)\geq R_{k}$ such that $w=\tG_n(z_0)$. 
Hence 
\[|\tilde{\Phi}_n(w)-w-C_n|=|z_0-\tG_n(z_0)-C_n|<\epsilon,\]
which implies that
\[\PP\left(\sup_{z\in \CC_{\parsig+}, n\leq N} 
|\tilde{\Phi}_n(z)-z-C_n|<\epsilon\right)\geq 
\PP(\Omega_0(N,\epsilon,\parsig)).\]

Since
\begin{equation}
\PP(\Omega_0(N, \epsilon, \parsig))\geq 1-\sum_{m=1}^{M}\PP(\Omega_{R_{m}}^c),
\label{totproblowerest}
\end{equation}
it will be sufficient to establish the upper bound on $\PP(\Omega_{R_{m}}^c)$ stated below in \eqref{omegaRM}, from which the result will follow.
\end{proof}

We now prove the sequence of lemmas needed to establish this bound. % on $\PP(\Omega_{R_{m}}^c)$ stated in \eqref{omegaRM}

\begin{lemma}\label{pointcontrol}
For $z \in \ell_R$ and with
\[M_n(z)=\tG_n(z)-z+C_n, \quad n\geq 1,\]
the stopped process $(M_{n}^{T_{R_m}-1}(z), n\geq 1)$ is a martingale with respect to the filtration $\{\mathcal{F}_n\}_{n=1}^{\infty}$.
\end{lemma}
\begin{proof}
Measurability is immediate from the definitions, and integrability follows from the boundedness of 
$M_n$. 

Our first observation is that 
\[T_{R_m}(\omega)=n_m, \quad \omega \in \Omega_{R_m},\]
for all $m=1,\ldots,M$; this follows from the inequality
\[\epsilon>|\tG_{T_{R_m}}(z)-z+C_{T_{R_m}}|\geq \textrm{Re}(z)-C_{T_{R_m}}-
\textrm{Re}(\tG_{T_{R_m}}(z)).\]

We write
\[M_{n+1}(z)-M_n(z)=\tG_{n+1}(z)+c_{n+1}=\tilde{g}_{c_{n+1}}(\tG_n(z)-i\theta_{n+1})+c_{n+1}.\]
For $n\leq T_{R_m}-1$, by the definitions of the stopping time and $R_m$, 
\begin{equation}
\textrm{Re}(\tG_n(z))>R_m-C_n-\epsilon \geq \parsig-3\epsilon\geq d> d_{n+1}.
\label{martinglowerbd}
\end{equation} 
Consequently we can apply Lemma~\ref{gintegral}, and we obtain
\[\EE[M_{n+1}(z)-M_n(z)|\mathcal{F}_n]=\frac{1}{2\pi}\int_0^{2\pi}\tilde{g}_{c_{n+1}}(\tG_n(z)-i\theta)d\theta+c_{n+1}=0,\]
which establishes that $\EE[M_{n+1}(z)|\mathcal{F}_n]=M_n(z)$. 
\end{proof}

We next show that, for a single fixed $z\in \ell_{R_m}$, the process $|M_n(z)|$ is small with high probability. We make use of the following version of Bernstein's inequality (see \cite[Proposition 1]{F75}).
\begin{lemma}\label{Bernstein}
Let $\{x_j\}_{j=1}^{\infty}$ be a martingale difference sequence with respect to the filtration $\{\mathcal{F}_j\}_{j=1}^{\infty}$. Suppose $|x_j|\leq A$ for all $j=1,2,\ldots$, set 
\[M_n=\sum_{j=1}^nx_j, \quad n\geq 1,\]
and define
\[\langle M\rangle_n=\sum_{j=1}^{n}\EE[|x_j|^2|\mathcal{F}_{j-1}].\]

Then, for $\epsilon>0$ and any bounded stopping time $\tau$,
\begin{equation}
\PP\left(\sup_{n\leq \tau}|M_n|>\epsilon, \langle M\rangle_{\tau}\leq L\right)
\leq 2\exp\left(-\frac{\epsilon^2}{2\left(L+\frac{A\epsilon}{3}\right)}\right).
\end{equation}
\end{lemma}

\begin{lemma}
\label{expmart}
Let $\epsilon>0$ and suppose $\parsig-3\epsilon>d$. 
Then, for any $z \in \ell_{R_m}$,
\begin{equation}
\PP\left(\sup_{n\leq T_{R_m}}|M_n(z)|\geq \epsilon\right)
\leq 2 \exp\left(-A(\parsig, \epsilon, m)\frac{\epsilon^2}{\capc}\right), 
\end{equation}
where 
\begin{equation}
A(\parsig, \epsilon, m)=\frac{A}{\log \frac{1}{\parsig-3\epsilon}+\log(\parsig+m\epsilon)+\frac{\epsilon}{\parsig-\epsilon}}
\end{equation}
and $A>0$ is a universal constant.
\end{lemma}
\begin{proof}
We first establish that the increments $x_{n}(z)=|M_{n+1}(z)-M_n(z)|$ are uniformly bounded, and tend to zero with $\capc$. Indeed, when $n\leq T_{R_m}-1$, the estimate \eqref{martinglowerbd} holds, and using the asymptotic expansion of the
map $\tilde{g}_{c_{n+1}}$, or \cite[Equation (4)]{NT12}, we obtain
\begin{align*}
|M_{n+1}(z)-M_n(z)|&=|\tilde{g}_{c_{n+1}}(\tG_n(z)-i\theta_{n+1})+c_{n+1}|\\
&\leq  A_3\frac{c_{n+1}}{e^{\textrm{Re}(\tG_n(z))}-1}\\& \leq A_3\frac{2c_{n+1}}{R_m-C_n-\epsilon}\\
&\leq  A_3\frac{2\capc}{\parsig-3\epsilon}
\end{align*}
Next, we turn to second moments; applying the same bounds as before, we find that
\begin{align*}
\EE[|M_{n+1}(z)-M_n(z)|^2|\mathcal{F}_n]&\leq A_3^2\frac{(c_{n+1})^2}{2\pi}\int_0^{2\pi}\frac{1}{|e^{\textrm{Re}(\tG_n(z))-i\theta}-1|^2}d\theta \\
&=A_3^2\frac{(c_{n+1})^2}{e^{2\textrm{Re}(\tG_n(z))}-1}\\
&\leq A_3^2\frac{(c_{n+1})^2}{R_m-C_n-\epsilon}.
\end{align*}
To obtain the second equality, we compute the integral $\int_{\TT}|e^{\sigma+i\theta}-1|^{-2}d\theta$ explicitly. Using this estimate, we obtain that
\begin{align*}
\sum_{j=1}^{T_{R_m}-1}\EE[x_{j}^2|\mathcal{F}_j]&\leq A_3^2\sum_{j=1}^{n_m}\frac{(c_{j+1})^2}{R_m-C_j-\epsilon}\\
&\leq A_3^2\capc \int_0^{R_m-\parsig+2\epsilon}\frac{1}{R_m-x-\epsilon}dx\\
&\leq A_3^2 \capc \left ( \log\frac{1}{\parsig-3\epsilon}+ \log(R_m-\epsilon) \right ).
\end{align*}
We now invoke Lemma \ref{Bernstein}, and the desired exponential bound follows.
% , as the denomintor in the right-hand side satisfies
% \[2\capc \log\frac{1}{\parsig-3\epsilon}+\capc \log(R_m-\epsilon)+\frac{\epsilon}{3}\frac{A_1\capc}{\parsig-\epsilon}\geq 
% 2\capc \log\left(1+\frac{m}{\parsig-3\epsilon}\right)
% \rightarrow 0, \quad \capc\rightarrow 0.\]
\end{proof}

We proceed by showing that the martingales $(M_n(z))$ do not vary too much 
over the line $\ell_R$.
\begin{lemma}\label{linecontrol}
Let $L \in \NN$ be given. Then, provided $\capc (\sigma-3 \epsilon)^{-3} \leq 1$, 
%\bg(can replace 1 by a larger absolute constant if necessary)\eg,\er 
\begin{multline}
\PP\left(\sup_{n\leq T_{R_m}-1}|M_n(z)-M_n(w)|\geq \frac{\epsilon}{2} \ 
\textrm{for some}\ z,w\in \ell_{R_m}, \ |z-w|\leq \pi/L \right)\\
\leq A \frac{\capc}{L^{2/3}\epsilon^2[(\parsig-3\epsilon)^3 \wedge 1 ]},
\end{multline}
where $A>0$ is a universal constant.
\end{lemma}
\begin{proof}
We again argue as in \cite[Section 5]{NT12}. 

We set $\tilde{M}_n(z,w)=M_n(z)-M_n(w)$; then by a Lipschitz-type estimate, and 
the bounds in \cite{NT12},
\begin{align*}|\tilde{M}_{n+1}(z,w)|&=|\tilde{g}_{c_{n+1}}(\tG_n(z)-i\theta_{n+1})-
\tilde{g}_{c_{n+1}}(\tG_n(w)-i\theta_{n+1})|\\&
\leq \|\tilde{g}'_{c_{n+1}}\|_{\infty}\left(|\tilde{M}_n(z,w)|+|z-w|\right)\\
&\leq A_4\frac{c_{n+1}}{\min\{\textrm{Re}(\tG_n(z)-1),(\textrm{Re}(\tG_n(z)-1))^2\}} \left(|\tilde{M}_n(z,w)|+|z-w|\right).
\end{align*}
To lighten notation, we set $q(s)=\min\{s,s^2\}$.
Mimicking the Gr\"onwall-type argument in \cite{NT12} in the present setting 
leads to
\begin{align*}
\EE\left[\sup_{k\leq T_{R_m}}|\tilde{M}_k(z,w)|^2\right]&\leq |z-w|^2\left(\exp\left(A_4\capc \int_{\parsig-3\epsilon}^{\infty}\frac{ds}{[q(s)]^2}\right)-1\right).
\end{align*}
We now compute
\[\capc\int_{\parsig-3\epsilon}^{\infty}\frac{1}{[q(s)]^2}ds
=\capc\int_{\parsig-3\epsilon}^1\frac{ds}{s^4}+\capc\int_1^{\infty}\frac{ds}{s^2}
=\frac{\capc}{3}\left(3+\frac{1}{(\parsig-3\epsilon)^3}\right).\]
Note that
\[\frac{\capc}{3}\left(3+\frac{1}{(\parsig-3\epsilon)^3}\right) \leq 2 \capc [(\parsig-3\epsilon)^{-3} \vee 1 ] \leq 2.\]
So, using the inequality 
\[e^x-1
\leq x\left(1+\frac{x}{2}e^{x}\right),\quad x\geq 0,\] 
we deduce
\[\EE\left[\sup_{k\leq T_{R_m}}|\tilde{M}_k(z,w)|^2\right]
\leq A_5\capc [(\parsig-3\epsilon)^{-3} \vee 1 ] |z-w|^2.\]
By Kolmogorov's theorem then (see \cite[Proposition 5.1]{NT12} for details), 
we have
\[\sup_{k\leq T_{R_m}}|\tilde{M}_k(z,w)|\leq 
V(\capc,\parsig, \epsilon)|z-w|^{1/3},\quad z,w\in \ell_{R_m},\]
for a random variable $V$ with second moment bounded in terms of 
$\capc [(\parsig-3\epsilon)^{-3} \vee 1 ]$.
The desired conclusion now follows upon applying Chebyshev's inequality.
\end{proof}
% \begin{remark}
% For future reference, we record that if $\parsig$ is small, then
% \[\capc\int_{\parsig-3\epsilon}^{\infty}\frac{1}{[q(s)]^2}ds
% =\capc\int_{\parsig-3\epsilon}^1\frac{ds}{s^2}+\capc\int_1^{\infty}\frac{ds}{s^4}
% =\frac{\capc}{3}\left(2+\frac{1}{(\parsig-3\epsilon)^3}\right);\]
% this suggest we might have to sacrifice an additional 
% power of $\capc$ here as we need $\capc/\parsig^3$ to tend to zero.
% \end{remark}
To achieve uniform control over the process for all starting points 
on $\ell_R$, we pick points $z_k\in \ell_R$ with spacing $2\pi/L$. Then, 
after combining the previous two lemmas, and using union bounds, we deduce that
\[\PP(\Omega_{R_m}^c)\leq 2Le^{-A(\parsig,\epsilon, m)\frac{\epsilon^2}{\capc}}
+A\frac{\capc}{L^{2/3}\epsilon^2[(\parsig-3\epsilon)^3\wedge 1]}.\]
Optimizing $L\mapsto K_1L+K_2/L^{2/3}$ over $L$, we obtain
\begin{equation}
\PP(\Omega_{R_m}^c)\leq A\frac{\capc^{3/5}}{\epsilon^{6/5}[(\parsig-3\epsilon)^{9/5}\wedge 1]}
e^{-A(\parsig,\epsilon, m)\frac{\epsilon^2}{\capc}},
\label{omegaRM}
\end{equation}
which then establishes Theorem \ref{NTmain}.

Switching back from logarithmic to standard coordinates, we can interpret the theorem as follows.

\begin{corollary}
\label{NTstandcoord}
Under the same hypotheses as Theorem \ref{NTmain}, on the event $\{ \sup_{z\in \CC_{\parsig+}, n\leq N} 
|\tilde{\Phi}_{n}(z)-z-C_{n}|<\epsilon \}$, we have
\[\sup_{z\in \CC_{\parsig+}, n\leq N}|\Phi_{n}(z)-e^{C_{n}}z|<\epsilon e^{6\epsilon}, \]
and
\[\sup_{z\in \CC_{\parsig+}, n\leq N}|\Phi_{n}'(z)-e^{C_{n}}| <\frac{2 \epsilon e^{6\epsilon}}{e^{\parsig}-1}.\]
Hence the probability of these events is bounded below by
\[1-AN\frac{\capc^{8/5}}{\epsilon^{11/5}[(\parsig-3\epsilon)^{9/5}\wedge 1]}\exp\left(-A(\parsig, \epsilon, N) \frac{\epsilon^2}{\capc}\right).\]
\end{corollary}
\begin{proof}
The first claim is immediate from standard estimates on the exponential function. For the second, we can use Cauchy's formula with $|z|\geq e^{\parsig}$ to get an estimate on the derivative by observing that 
\begin{align*}
|\Phi_{n}'(z)-e^{C_{n}}|&\leq \frac{1}{2\pi}\int_{\TT(z,\frac{1}{2}(e^{\parsig}-1))}
\frac{|\Phi_{n}(\zeta)-e^{C_{n}}\zeta|}{|z-\zeta|^2}|d\zeta|\\
&<\frac{2 \epsilon e^{6\epsilon}}{e^{\parsig}-1}.
\end{align*}
\end{proof}

Theorem \ref{NTmain} applies in particular to sequences of capacities $\{c^*_{k+m} \}_{m=1}^{\infty}$ defined in \eqref{detcaps} and uniform (iid) angles $\{\theta_{k+m}\}_{m=1}^{\infty}$ for any $k \in \mathbb{N}$. Before applying this theorem, we observe that 
\begin{align}
\left | C^*_{n,k}-\frac{1}{\alpha}\log\left(\frac{1+\alpha\capc n}{1+\alpha\capc (k-1)}\right) \right | 
&= \left | \sum_{m=k}^nc^*_m - \int_{k-1}^{n} \frac{\capc}{1+\alpha \capc x} dx \right | \nonumber \\
&\leq \alpha \capc^2 (n-k+1).
\label{ceestarcomparison}
\end{align}

%Thus, taking $N=\lfloor T/ \capc \rfloor$ for some given $T>0$, we obtain
%\[\lim_{\capc\rightarrow 0}\textrm{Cap}(K^*_n)=\lim_{\capc \rightarrow 0}e^{C^*_N}=(1+\alpha T)^{1/\alpha}.\]
%As $\alpha\rightarrow 0$, we recover $\textrm{Cap}(K_T)=\exp(T)$, the capacity growth rate in the standard
%$\textrm{HL}(0)$ model, which has been analyzed in detail in \cite{RZ05} and \cite{NT12}.

Combining this result with Theorem \ref{NTmain} and Corollary \ref{NTstandcoord} we get the following corollary immediately.

\begin{corollary}\label{logmapconvergence}
For any $N\in \mathbb{N}$, $\parsig>0$, $\capc >0$, $\alpha>0$ and $\epsilon > 0$ with $0 <\epsilon - \alpha \capc^2 N<\parsig/3$ and $\parsig-3\epsilon+3\alpha \capc^2 N>d\vee \capc^{1/3}$, we have
\begin{multline*}
\PP\left(\sup_{z\in \CC_{\parsig+}, k<n\leq N} 
\left |\tilde{\Phi}^*_{n,k}(z)-z-\frac{1}{\alpha}\log\left(\frac{1+\alpha\capc n}{1+\alpha\capc (k-1)}\right)\right|<\epsilon\right)\\
\geq 1-AN^2\frac{\capc^{8/5}}{(\epsilon-\alpha \capc^2 N)^{11/5}[(\parsig-3\epsilon+3\alpha \capc^2 N)^{9/5}\wedge 1]}\exp\left(-A(\parsig, \epsilon-\alpha \capc^2 N, N) \frac{(\epsilon-\alpha \capc^2 N)^2}{\capc}\right).
\end{multline*}
In standard coordinates, this means that on the event above
\[\sup_{z\in \CC_{\parsig+}, k<n\leq N} \left | \Phi^*_{n,k}(z)-\left (\frac{1+\alpha\capc n}{1+\alpha\capc (k-1)} \right )^{1/\alpha} z \right | < \epsilon e^{6 \epsilon}\]
and
\[\sup_{z\in \CC_{\parsig+}, k<n\leq N} \left | (\Phi^*_{n,k})'(z)-\left (\frac{1+\alpha\capc n}{1+\alpha\capc (k-1)} \right)^{1/\alpha} \right | < \frac{2 \epsilon e^{6\epsilon}}{e^{\parsig}-1}.\]
% \bg (Can improve slightly by letting $\parsig, \epsilon$ depend on $N'$ when applying the theorem above but only worth the extra untidyness if it would actually be useful). \eg
\end{corollary}

\section{Convergence of capacities}\label{capacityconvergence}
From this point onwards, we take the sequence $\{ c_k \}_{k=1}^{\infty}$ to be that specified by \eqref{sigregdef}. In this section we use the coupling between the starred and un-starred maps $(\Phi^*_n)'$ and $\Phi_n'$, induced by using the same sequence of angles of rotation, to show that with high probability the sequence $\{ c_k \}_{k=1}^{\infty}$ is close to $\{ c^*_k \}_{k=1}^{\infty}$.

Since 
\[c_{n}=\frac{\capc}{|\Phi_{n-1}'(e^{\parsig+i\theta_{n}})|^{\alpha}},\] 
we can write
\begin{align*}
\left |\log \frac{c_{n}}{c^*_{n}} \right |&=\left|\log \frac{c_{n}}{\capc|[\Phi^*_{n-1}]'(e^{\parsig+i\theta_{n}})|^{-\alpha}}
+\log \frac{\capc|[\Phi^*_{n-1}]'(e^{\parsig+i\theta_{n}})|^{-\alpha}}{c^*_{n}}\right|\\
&\leq \alpha \left| \log \frac{|\Phi'_{n-1}(e^{\parsig+i\theta_{n}})|}{|[\Phi^*_{n-1}]'(e^{\parsig+i\theta_{n}})|}\right|+\left|\log\frac{|[\Phi^*_{n-1}]'(e^{\parsig+i\theta_{n}})|^{\alpha}}{1+\alpha \capc (n-1)}\right|.
\end{align*}

We now show that $\Phi_n'$, the derivatives of the 
$\textrm{HL}(\alpha,\parsig)$ maps, are close to $(\Phi^*_n)'$ away from the boundary. 
We define accumulated errors  
%for 
%\[w \in e^{\parsig}\Delta,\]
%=\{e^{\parsig}w\colon w\in \Delta\}
\begin{equation}
S_{\alpha}(w,n)=\log \left|\frac{\Phi'_{n}(w)}{(\Phi^*_{n})'(w)}\right|
\label{prevstep}
\end{equation}
and similarily, we introduce the single-step errors
\begin{equation}
s_{\alpha}(c_n,c^*_n,w) = \log \left|\frac{f'_n(w)}{(f_n^*)'(w)}\right|.
\label{capstep}
\end{equation}
Note that by Lemma~\ref{capreplacement}.  
\[ |s_{\alpha}(c_n,c^*_n,w)| \leq A_0\frac{|c_n-c^*_n|}{(|w|-1)^{2}}.\]
%We obtain a recursion for $S_n$ and bound it using the above inequality.

We express $\log |\Phi'_n(z)|$ in terms of $\log |(\Phi^*_n)'(z)|$ and 
the error function introduced above, and find that
\begin{align*}
\log |\Phi'_{n}(z)|&=\log |f'_n(z)|+\log|\Phi'_{n-1}(f_n(z))|\\
= &\,\log |f'_n(z)|+\log |\Phi'_{n-1}(f^*_n(z))| + \log 
\left|\frac{\Phi'_{n-1}(f_n(z))}{\Phi'_{n-1}(f^*_n(z))}\right|\\
 =& \,\log |f'_n(z)|+\log |(\Phi^*_{n-1})'(f^*_n(z))|
+\log \left|\frac{\Phi'_{n-1}(f_n(z))}{\Phi'_{n-1}(f^*_n(z))}\right|+S_{\alpha}(f^*_n(z),n-1)\\
= & \,\log |(f^*_n)'(z)|+\log|(\Phi^*_{n-1})'(f^*_n(z))|
+\log \left|\frac{\Phi'_{n-1}(f_n(z))}{\Phi'_{n-1}(f^*_n(z))}\right|\\
&+S_{\alpha}(f^*_n(z),n-1)]+s_{\alpha}(c_n,c^*_n,s)\\
= &\,\log |(\Phi^*_{n})'(z)|
+\log \left|\frac{\Phi'_{n-1}(f_n(z))}{\Phi'_{n-1}(f^*_n(z))}\right|+S_{\alpha}(f^*_n(z),n-1)+s_{\alpha}(c_n,c^*_n,s),
% &=|f'_n(z)|^{-\alpha}|(\Phi^*_{n-1})'(f_n(z))|^{-\alpha}(1+S_{\alpha}(f_n(z),n-1))\\
% &=|f'_n(z)|^{-\alpha}|(\Phi^*_{n-1})'(f^*_n(z))|^{-\alpha}\\ &\times
% (1+S_{\alpha}(f_n(z),n-1))\left[\frac{|(\Phi^*_{n-1})'(f_n(z))|}{|(\PN(\eps)hi^*_{n-1})'(f^*_n(z))|}\right]^{-\alpha}\\
% &=|(f^*_n)'(z)|^{-\alpha}|(\Phi^*_{n-1})'(f^*_n(z))|^{-\alpha}\\
% &\times [1+S_{\alpha}(f_n(z),n-1)][1+s_{\alpha}(c_n,c^*_n,z)]\left[\frac{|(\Phi^*_{n-1})'(f^*_n(z))|}{|(\Phi^*_{n-1})'(f_n(z))|}\right]^{\alpha}\\
% &=|(\Phi^*_{n})'(z)|^{-\alpha}\\
% &\times [1+S_{\alpha}(f_n(z),n-1)][1+s_{\alpha}(c_n,c^*_n,z)]\left[\frac{|(\Phi^*_{n-1})'(f^*_n(z))|}{|(\Phi^*_{n-1})'(f_n(z))|}\right]^{\alpha}
\end{align*}
and hence, 
\[
S_{\alpha}(z,n) = S_{\alpha}(f^*_n(z),n-1)+s_{\alpha}(c_n,c^*_n,z) + \log \left|\frac{\Phi'_{n-1}(f_n(z))}{\Phi'_{n-1}(f^*_n(z))}\right|.
\]

We bound the ratio of derivatives of composed maps in terms of $|c_n-c^*_n|$ and hence show that the right-hand side can be bounded 
in terms of $|c_n-c^*_n|$ and $S_{\alpha}(\cdot, n-1)$.

% By definition, we have
% \[|s_{\alpha}(c_n,c^*_n,s)|=|(f^*_n)'(z)|^{\alpha}
% \left|\frac{1}{|f'(z)|^{\alpha}}-\frac{1}{|(f^*_n)'(z)|^{\alpha}}\right|.\]
% Note that \footnote{Obtain good estimate.} $|(f^*_n)'(z)|\leq C(e^{\parsig}-1)^{-1/2}$.
% We now apply Lemma \ref{capreplacement} and use the fact that $c^*_k<\capc$.
% This yields
% \begin{align*}|s_{\alpha}(c_n,c^*_n,z)|&=|(f^*_n)'(z)|^{\alpha}
% \left|\frac{1}{|f'_n(z)|^{\alpha}}-\frac{1}{|(f^*_n)'(z)|^{\alpha}}\right|\\
% & \leq \alpha D(\alpha, \parsig)\frac{|c_n-c^*_n|}{(e^{\parsig}-1)^{\alpha/2}}
% +\mathcal{O}\left(\frac{c_n^2+(c^*_n)^2}{(e^{\parsig}-1)^{\alpha/2+3}}\right),
% \end{align*}
% with
% \[C_1(\alpha,\parsig)=C \alpha D(\parsig)/(e^{\parsig}-1)^{\alpha/2}.\]

% %17\alpha(1+\capc)\left(\frac{e^{2\parsig}}{e^{\parsig}+1}\right)^{\alpha}\]
% %and is well-behaved as a function of $\parsig$. 

\begin{lemma} 
Suppose that $\capc > 0$ and $z \in \Delta$ satisfies $(|z|-1)^2 > 16 A_1 \capc$ % \bg ( actually need $\frac{2 A_1 \capc}{(e^\parsig-1)^2} \leq 1/2, e^\parsig - 1 > 2d(\capc), \parsig^2 >2 \capc$ but I think the condition I've given covers all this \eg)
and suppose that $n < \inf \{ k : |c_k-c_k^*| > 2 \capc \}$. Then
\begin{equation}
\left | \log \left|\frac{\Phi'_{n-1}(f^*_n(z))}{\Phi'_{n-1}(f_n(z))}\right| \right | 
\leq 8 A_1\frac{|c_n-c^*_n|}{(|z|-1)^2}.
\end{equation}
\end{lemma}
\begin{proof}
Suppose $r=|f^*_n(z)|-1>|f_n(z)|-1$; we deal with the reverse case later. 
First observe that 
$|f_n(z)|\in B(f^*_n(z),r)$ so that we can ``localize'' the standard distortion estimates on the unit disk $\DD$ to 
$B(f_n^*(z),r)$. To do this we write
\begin{equation}
f_n(z)=f^*_{n}(z)
+(|f^*_n(z)|-1)\left[\frac{f_n(z)-f^*_n(z)}{|f^*_n(z)|-1}\right]=w^*+rw(z),
\label{renorm}
\end{equation}
and use Lemma \ref{capreplacement}, our assumptions on $|z|$ and $\capc$ and the crude estimate
\[|f^*_n(z)|-1\geq |z|-1\]
to bound the denominator appearing in $w(z)$ to see that 
\begin{align}
\label{wbound}
 |w(z)| &\leq A_1\frac{|c_n - c^*_n|}{(|z| -1)^2} \\
        &\leq 1/2. \nonumber
\end{align}

In these new coordinates,
\[\frac{(1-|w(z)|)^3}{1+|w(z)|}\leq \left|\frac{\Phi'_{n-1}(w^*)}{\Phi'_{n-1}(f_n(z))}\right|\leq 
\frac{(1+|w(z)|)^3}{1-|w(z)|},\]
so that
\[3\log(1-|w(z)|)-\log(1+|w(z)|) \leq \log \left|\frac{\Phi'_{n-1}(w^*)}{\Phi'_{n-1}(f_n(z))}\right| \leq 
3\log(1+|w(z)|)-\log(1-|w(z)|).
\]

Setting $F(x)=3\log(1+x)-\log(1-x)$, we note that
\[|F(x)|\leq \sup_{x\in [-1/2,1/2]}|F'(x)| \cdot |x| \leq 8|x| , \quad x\in [-1/2,1/2],\]
so, using \eqref{wbound},
\begin{equation*}
\left | \log \left|\frac{\Phi'_{n-1}(w^*)}{\Phi'_{n-1}(f_n(z))}\right| \right | \leq
8|w(z)| \\
\leq 8 A_1\frac{|c_n-c^*_n|}{(|z|-1)^2}.
\end{equation*}
%+\mathcal{O}\left(\frac{c_n^2+(c^*_n)^2}{(e^{\parsig}-1)^4}\right).\] 
% After expanding $(1+x)^{\alpha}$, we finally arrive at
% \begin{equation}
% \left|\frac{\Phi'_{n-1}(f^*_{n}(z))}{\Phi'_{n-1}(f_n(z))}\right|^{\alpha}
% \leq 1+A_3(\alpha, \parsig)|c_n-c^*_n|
% +\mathcal{O}\left(\frac{c_n^2+(c^*_n)^2}{(e^{\parsig}-1)^4}\right),
% \end{equation}
% where we can take $A_3(\alpha, \parsig)=4\alpha M(\parsig)/(e^{\parsig}-1)$.

Returning to the case $|f^*_n(z)|-1<|f_n(z)|-1$, we interchange the roles of 
$f_k$ and $f^*_k$ in \eqref{renorm}. In that case we get exactly the same bound by observing that 
\[|f_n(z)|-1\geq|f^*_n(z)|-1\geq |z|-1.\]
This completes the proof.
\end{proof}

We return to the relation
\begin{align*}
S_{\alpha}(z,n) = S_{\alpha}(f^*_n(z),n-1)+s_{\alpha}(c_n,c^*_n,z) + \log \left|\frac{\Phi'_{n-1}(f_n(z))}{\Phi'_{n-1}(f^*_n(z))}\right|\end{align*}
and implement our lemmas to obtain 
\begin{align*}
|S_{\alpha}(z,n) | & \leq |S_{\alpha}(f^*_n(z),n-1)|
+A_0\frac{|c_n-c^*_n|}{(|z|-1)^2}
+8A_1\frac{|c_n-c_n^*|}{(|z|-1)^2} \\
&= |S_{\alpha}(f^*_n(z),n-1)|+A_3\frac{|c_n-c_n^*|}{(|z|-1)^2}
\end{align*}
with $A_3=A_0+8A_1$. 
Thus, we have
\[|S_{\alpha}(z,n)|\leq A_3 \sum_{k=1}^{n-1}
\frac{|c_k-c^*_k|}{(|\Phi^*_{n-1,k}(z)|-1)^2}.\]

In what follows, let $N \in \mathbb{N}$, and define the stopping time 
\begin{equation*}
N(\eps) = \inf \left \{ n : \left | \log \frac{c_n}{c_n^*}\right | \geq 1 \mbox{ or } \sup_{z\in \CC_{\parsig+}, k < n} \left |\Phi^*_{n,k}(z)-\left (\frac{1+\alpha\capc n}{1+\alpha\capc (k-1)}\right)^{1/\alpha} z \right | \geq \epsilon \right \} \wedge N. 
\end{equation*}
Note that for $k \leq N(\eps)-1$, 
\begin{align*}
|c_k-c^*_k|&\leq c_k^* \left | \log \frac{c_k}{c_k^*}\right | \left (1+ \left | \log \frac{c_k}{c_k^*}\right |\Big/2 \right ) \\ 
&\leq 2 c_k^* \left | \log \frac{c_k}{c_k^*}\right | \\
&< 2 \capc
\end{align*}
and hence $c_k < 3 \capc$.

\begin{theorem}\label{capincrconv}
Suppose that $\capc, \parsig, \epsilon>0$ satisfy $\parsig^2 > 16 A_1 \capc$ and $0<\epsilon < \parsig/2$. Then
\[
 \sup_{n \leq N(\eps)} \left|\log \frac{c_n}{c^*_n}\right|\leq 
\frac{2 \alpha \epsilon}{\parsig}
\exp\left(\frac{8A_3 \log (1+\alpha \capc N)}{\parsig^2}\right).
\]
In particular, taking $\epsilon=\capc^{1/3}$ and $\parsig > (64 A_3 \log(1+ \alpha T))^{1/2} (\log \capc^{-1})^{-1/2}$, then
\[
 \mathbb{P} \left ( \sup_{n \leq \lfloor T / \capc \rfloor } \left|\log \frac{c_n}{c^*_n}\right| > \alpha \capc^{1/6} \right ) \to 0
\]
as $\capc \to 0$. % \bg Absolute constants have not been optimised but seem good enough for our purposes. \eg
\end{theorem}
\begin{proof}
Suppose that $2 \leq n \leq N(\eps)$. We recall our decomposition
\begin{align*}
\left |\log \frac{c_{n}}{c^*_{n}} \right |
\leq &\, \alpha \left| \log \frac{|\Phi'_{n-1}(e^{\parsig+i\theta_{n}})|}{|[\Phi^*_{n-1}]'(e^{\parsig+i\theta_{n}})|}\right|+\left|\log\frac{|[\Phi^*_{n-1}]'(e^{\parsig+i\theta_{n}})|^{\alpha}}{1+\alpha \capc (n-1)}\right| \\
=&\, \left|\log\frac{|[\Phi^*_{n-1}]'(e^{\parsig+i\theta_{n}})|^{\alpha}}{1+\alpha \capc (n-1)}\right| + \alpha S_{\alpha}(e^{\parsig+i\theta_{n}},n) \\
\leq &\, \left|\log\frac{|[\Phi^*_{n-1}]'(e^{\parsig+i\theta_{n}})|^{\alpha}}{1+\alpha \capc (n-1)}\right| + \alpha A_3 \sum_{k=1}^{n-1}
\frac{|c_k-c^*_k|}{(|\Phi^*_{n-1,k}(e^{\parsig+i\theta_{n}})|-1)^2} \\
\leq &\, \left|\log\frac{|[\Phi^*_{n-1}]'(e^{\parsig+i\theta_{n}})|^{\alpha}}{1+\alpha \capc (n-1)}\right|  + 2 \alpha A_3 \sum_{k=1}^{n-1}
\frac{c^*_k \left | \log \frac{c_k}{c_k^*}\right |}{(|\Phi^*_{n-1,k}(e^{\parsig+i\theta_{n}})|-1)^2}. 
\end{align*}

Now since $|\Phi^*_{n-1,k}(z)-\left ( \frac{1+\alpha\capc (n-1)}{1+\alpha\capc (k-1)} \right )^{1/\alpha}z|<\epsilon$, we have
\begin{align*}
|\Phi^*_{n-1,k}(e^{\parsig+i\theta_{n}})|-1 &> \left ( \frac{1+\alpha \capc (n-1) }{1+\alpha \capc (k-1) } \right ) ^{1/\alpha} e^{\parsig} -\epsilon -1 \\
&> 1\cdot(1+\parsig) - \epsilon -1 \\
&> \parsig / 2.
\end{align*}

Also, $|(\Phi^*_{n-1})'(z)-(1+\alpha \capc (n-1))^{1/\alpha}|<2\epsilon/(e^\parsig-1)$, and so 
$$
\left|\log\frac{|[\Phi^*_{n-1}]'(e^{\parsig+i\theta_{n}})|^{\alpha}}{1+\alpha \capc (n-1)}\right| < \frac{2 \alpha \epsilon}{\parsig}.
$$

Hence 
\begin{align*}
\left |\log \frac{c_{n}}{c^*_{n}} \right | 
&< \frac{2 \alpha \epsilon}{\parsig}  + \frac{8 A_3 \alpha}{\parsig^2}  \sum_{k=1}^{n-1}
c^*_k \left | \log \frac{c_k}{c_k^*}\right | \\
&< \frac{2 \alpha \epsilon}{\parsig}  + \frac{8 A_3 \alpha \capc}{\parsig^2} \sum_{k=1}^{n-1}
(1+\alpha \capc (k-1))^{-1}\left | \log \frac{c_k}{c_k^*}\right |. % \\
% &<& \alpha \left (\frac{2 \epsilon}{\parsig} + \capc T \right ) + 32 C_4 \alpha \capc \parsig^{-2} \sum_{k=1}^{n}
% \left | \log \frac{c_k}{c_k^*}\right |, 
\end{align*}
An application of Gr\"onwall's lemma yields 
\begin{align*}
\left|\log \frac{c_{n}}{c^*_{n}}\right| &\leq \frac{2 \alpha \epsilon}{\parsig} \exp \left ( \frac{8 A_3 \alpha \capc}{\parsig^2} \int_0^n (1+\alpha \capc x)^{-1} dx \right ) \\
&\leq \frac{2 \alpha \epsilon}{\parsig} \exp \left ( \frac{8 A_3 \log (1+\alpha \capc N)}{\parsig^2} \right ).
\end{align*}
%\bg Instead of using Gronwall, if we can find an ``exact'' solution to the recu%rrence relation we may be able to get $\parsig$ polynomial in $\capc$.\eg

% \begin{eqnarray*}
% \left |\log \frac{c_{n+1}}{c^*_{n+1}} \right | &\leq& \alpha \left (\frac{2 \epsilon}{\parsig} + \capc T \right ) \exp (32 C_4 \alpha \capc n \parsig^{-2})\\ &\leq& \alpha \left (\frac{2 \epsilon}{\parsig} + \capc T \right ) \exp (32 C_4 T \parsig^{-2}).
% \end{eqnarray*}
The second statement follows by applying Corollary \ref{logmapconvergence} to show that $N(\eps) > \lfloor T/\capc \rfloor $ with high probability.
\end{proof}

\section{Scaling limits for conformal maps}\label{mapconvergence}
%We have established in Theorem \ref{capincrconv} that the sequence of capacity increments $\{c_k\}_k$ defined in \eqref{sigregdef} is close to the 
%deterministic sequence $\{c^*_k\}_k$ defined in \eqref{detcaps}, in the sense that with high probability, for all $T>0$,
%\[|c_k-c^*_k|\leq 2 \alpha \capc^{1+\beta}, \quad k=1,\ldots, N_{T}\]
%for some $\beta>0$ as $\capc\to 0$, provided that $\parsig \gg (\log \capc^{-1})^{-1/2}$. 

We shall now prove that (for fixed $\alpha>0$) the conformal maps $\Phi_N$ associated with the regularized $\mathrm{HL}(\alpha,\parsig)$ 
process converge in probability, when $N=\lfloor T/\capc \rfloor$ for some fixed $T>0$ and $\capc \rightarrow 0$, to the deterministic limit map
\[\Psi_T(z)=(1+\alpha T)^{1/\alpha}z, \quad z\in \Delta. \]

As was indicated in Section \ref{prelim} (see also \cite{RZ05, JST12}), the 
composed maps $\Phi_N$ can be described 
using the Loewner equation \eqref{loewnereq} driven by the random measures 
\eqref{drivingmeasures}. 
To be precise, $\Phi_N = \varphi_{C_N}$ where $\varphi_t$ is the Loewner chain driven by the measures 
$\mu_\capc$ on the space $S=\TT\times [0, \infty)$ satisfying
\[d\mu_\capc(\theta,t)=\delta_{\xi_\capc(t)} dt,\]
where
\begin{equation}
\xi_\capc(t)=\exp\left(i\sum_{k=1}^{N}
\theta_k\mathbf{1}_{[C_{k-1}, C_k)}(t)\right), \quad t < C_N.
\end{equation} 
We denote by $\mathcal{M}(S)$ the collection of bounded Borel measures on the space $S$, and endow $\mathcal{M}(S)$ 
with the weak topology. That is, with the notation
\[\langle f,\nu\rangle=\int_Sf(\theta,t)d\nu(\theta,t),\]
we have
$\mathcal{M}(S)\ni \nu_n\rightarrow \nu$ provided $\langle f,\nu_n\rangle
\rightarrow \langle f,\nu\rangle$ for every $f\in C_b(S)$. Thus 
the measures $\mu_\capc$ associated with the 
$\mathrm{HL}(\alpha,\parsig)$ growth process can be viewed as random elements of the space $\mathcal{M}(S)$.
Note that the pairing of $f\in C_b(S)$ with these random measures produces random variables determined by the 
expressions 
\begin{equation}
\langle f, \mu_\capc\rangle=\sum_{k=1}^{N}\int_{C_{k-1}}^{C_k}f(\theta_k,t)dt.
\label{pairingexpr}
\end{equation}

\begin{lemma}
\label{muconvlemma}
For fixed $T>0$, let $N= \lfloor T/\capc \rfloor$, let $(\mu_{\capc})_{\capc>0}$ be the measures defined above that generate $\mathrm{HL}(\alpha, \parsig)$ maps $\Phi_N$, and let $\nu\in \mathcal{M}(S)$ be the measure given by 
\[d\nu(\theta,t)=\mathbf{1}_{[0,\log(1+\alpha T)/\alpha]}(t) d\theta dt.\] 

Then, provided $\parsig \gg (\log \capc^{-1})^{-1/2}$, we have $\mu_{\capc} \to \nu$
in distribution as $\capc \rightarrow 0$, with respect to the weak topology.
\end{lemma}
%(\bg Definition of $N(\eps)$ prevents it getting too large when $\alpha$ is small. In this section we don't want to let $\alpha\to 0$ so $\alpha\capc>\capc^{3/2}$ once $\capc$ small enough. \eg)
\begin{proof}
Let 
\[N' = \inf\{n : |c_{n+1}-c^*_{n+1}| > 2 \alpha \capc^{1+\beta}\} \wedge N  \]
where $\beta>0$ is the absolute constant guaranteed by Theorem \ref{capincrconv} such that $N' - N \to 0$ in probability
as $\capc \to 0$. It follows that $C_{N'} - C_N \to 0$ in probability.

Furthermore if $n \leq N'$, then 
\[|C_n-C^*_n|\leq\sum_{k=1}^n|c_k-c^*_k|\leq 2 \alpha T \capc^{\beta} \to 0\] and hence, by \eqref{ceestarcomparison},
\[
|C_{N'}  - \log(1+\alpha T) / \alpha| \leq |C_{N'}-C_N|+|C_N - C_N^*| + |C_N^* - \log(1+\alpha T)/\alpha| \to 0
\]
in probability as $\capc \to 0$.

Recall (see \cite[Theorem 16.16]{KallBook}) that in order to prove that $\mu_\capc \rightarrow \nu$ in distribution (and hence in probability since $\nu$ is a constant element of $\mathcal{M}(S)$) 
we have to show that $\langle f, \mu_\capc\rangle \to \langle f,\nu\rangle$ in distribution for each $f\in C_b(S)$.

Let $f\in C_b(S)$ be given. As $f$ is bounded,
\[\langle f, \mu_\capc\rangle - \sum_{k=1}^{N'}\int_{C_{k-1}}^{C_k}f(\theta_k,t)dt \to 0 \]
and
\[\langle f, \nu_\capc\rangle - \sum_{k=1}^{N'}\int_{C_{k-1}}^{C_k} \int_{\TT}f(\theta,t) d \theta dt \to 0 \]
in probability.
Also, by uniform continuity of $f$ on compact intervals, given $\varepsilon > 0$, for $\capc$ sufficiently small 
\[|f(\theta,t)-f(\theta,s)|\leq \varepsilon\]
whenever $|s-t|<3\capc$.
Therefore, the bound 
$c_k <3\capc$ for all $k \leq N'$ implies that
\[ \left|\sum_{k=1}^{N'}\int_{C_{k-1}}^{C_k} \left ( f(\theta_k, t)-f(\theta_k, C_{k-1}) \right ) dt\right| 
\leq \varepsilon \sum_{k=1}^Nc_k \leq 3 \varepsilon T \]
and similarly
\[ \left | \sum_{k=1}^{N'}\int_{C_{k-1}}^{C_k} \int_{\TT}( f(\theta, t)-f(\theta, C_{k-1}) ) d \theta dt \right | \leq 3 \varepsilon T.\]
Therefore, in order to show that
\[
|\langle f,\mu_\capc\rangle-\langle f, \nu\rangle| \to 0
\]
in probability, it is enough to show that
\begin{align}
\label{weaksetup}
&\left | \sum_{k=1}^{N'} \left ( c_k f(\theta_k,C_{k-1}) - \int_{\TT}c_k f(\theta,C_{k-1}) d \theta \right ) \right | \nonumber \\
&\leq \sum_{k=1}^{N'}|c_k f(\theta_k,C_{k-1}) - c^*_k f(\theta_k,C^*_{k-1})| 
+ \left |\sum_{k=1}^{N'} c^*_k \left ( f(\theta_k,C^*_{k-1}) - \int_\TT f(\theta,C^*_{k-1})d\theta \right ) \right | \nonumber \\
&\to 0  
\end{align}
in probability.
% converge weakly to the corresponding starred measures, and then use continuity 
% of the Loewner equation. Then, by our previous results, the starred maps 
% $\Phi^*_n$ converge to time-changed disk maps 
% $z\mapsto (1+\alpha T)^{1/\alpha}z$, 
% and this would complete the proof. 

% We sketch the arguments we have in mind in a Loewner proof. 
% Let $S=\TT \times [0,T]$, and let $f \in C_b(S)$ be given. 
% Then

We begin with the 
second term in \eqref{weaksetup}. Since the time increments $\{c_k^*\}$ are deterministic (and, in particular, independent of $\{\theta_k\}$), 
we can show that 
this term tends to zero almost surely by the strong law of large numbers, as in the proof of \cite[Theorem 2]{JST12}.  % (\bg This part of argument holds without 
% using that $N$ is a stopping time.\eg)

It remains to bound the first term in \eqref{weaksetup}.
For $k \leq N'$,
\begin{align*}
|c_kf(\theta_k,C_k)-c^*_kf(\theta_k,C_k^*)|&\leq 
c^*_k|f(\theta_k,C_k)-f(\theta_k,C^*_k)|+
|c_k-c^*_k||f(\theta_k,C_k)|\\
&\leq 2\capc \left[\max_{\theta \in \TT, \  |s-t|\leq \alpha T \capc^{\beta}}|f(\theta,s)-f(\theta,t)| + \alpha \capc^{\beta}\|f\|_{\infty} \right].
\end{align*}
Since $f$ is a bounded continuous function, and $N' \leq T/ \capc$, the sum 
tends to zero as $\capc\rightarrow 0$. Thus $\langle f, \mu_\capc \rangle \to \langle f, \nu \rangle$ in probability, and we are done.
% The first term can be made small by uniform continuity, 
% since $|C_k-C^*_k|<\epsilon$, and the second can also be made small 
% by choosing $\epsilon$ sufficiently small as $f$ is bounded. We 
% have $N$ such terms, but also a factor $\capc$ in front, and 
% hence the difference between integrals for an individual test function $f$ 
% can be made arbitrarily small, with high probability,
% by choosing a suitable $\epsilon$ in the previous Theorems.
\end{proof}

\begin{theorem}\label{MainThm1}
Let $T>0$ be fixed and set 
\[\Phi_{N}=f_1\circ\cdots\circ f_{N},\] 
with $N=\lfloor T/\capc \rfloor$. Then, provided $\parsig \gg (\log \capc^{-1})^{1/2}$, 
the maps $\Phi_N$ converge in distribution to the conformal map $\Psi_T(z)=(1+\alpha T)^{1/\alpha}z$ as $\capc\to 0$, with respect 
to the topology of uniform convergence on compact subsets. % \bg ($N(\eps)$ is now a stopping time but this is ok.)\eg
\end{theorem}

An almost identical argument can be used to show that, if $\alpha \to 0$ as $\capc \to 0$, then $\Phi_N(z) \to e^Tz$ in the same sense as above.
\begin{proof}
The map $\Phi_N = \varphi^{\capc}_{C_N}$ where $\{\varphi^{\capc}_t\}$ is the family of conformal mappings driven by the measures $\mu_\capc$. Since $\mu_{\capc} \to \nu$ in distribution with respect to the weak topology, by the continuity result contained in  
\cite[Proposition 1]{JST12}, we obtain 
\[\varphi^{\capc}_t(z) \to e^t z
\]
uniformly on compacts, for each $t \leq \log(1+\alpha T)/\alpha$. Since $C_N \to \log(1+\alpha T)/\alpha$ in probability as $\capc \to 0$, we get
\[
\Phi_N \to \exp \left ( \log(1+\alpha T)/\alpha \right ) z = \Psi_T(z)
\]
as required. 
% Finally, we appeal to \cite[Proposition 1]{JST12} to deduce convergence in 
% probability of the conformal maps: 
% $\Phi_{N(\eps)}\rightarrow (1+\alpha T)^{1/\alpha}z$ as $\capc\rightarrow 0$.
\end{proof}
Using the fact that $\int_{\TT}(z+\zeta)/(z-\zeta)|d\zeta|=1$, it is readily verified that the mappings 
\[\Psi_t(z)=(1+\alpha t)^{1/\alpha}z\]
solve the Loewner-type equation
\begin{equation}
\dot{\varphi}_t(z)=z\varphi'_t(z)\int_{\TT}\frac{z+\zeta}{z-\zeta}\frac{|d\zeta|}{|\varphi'_t(\zeta)|^{\alpha}}
\label{heleshaw}
\end{equation}
with initial condition $\varphi_0(z)=z$. The non-linear equation \eqref{heleshaw} is sometimes 
conjectured to be relevant for a description of 
possible small-particle scaling limits in the Hastings-Levitov 
$\mathrm{HL}(\alpha)$ model (see \cite{CM01, RZ05}). The case $\alpha=2$ 
corresponds to Hele-Shaw flow, which is known to be 
ill-posed in general, reflecting the fact 
that obtaining corresponding small-particle scaling for the 
conformal maps $\Phi_N$ when $\alpha>0$ and
$\parsig \ll d$ seems to be a challenging problem.
  
\begin{remark}
The arguments used in this section hold if $\{c_k\}_{k=1}^{\infty}$ is any sequence of capacities which is well approximated by a deterministic (or indeed independent of $\{\theta_k\}_{k=1}^{\infty}$) sequence of capacities. Specifically suppose there exists a sequence $\{c_k^*\}_{k=1}^{\infty}$ dependent on a parameter $\capc >0$ such that $0< c_k^* \leq \capc$ for all $k$, and for any $T>0$,
\[\mathbb{P} \left ( \inf\{n : |c_{n}-c^*_{n}| > 2 \alpha \capc^{1+\beta}\} < T/\capc \right ) \to 0\]
as $\capc \to 0$, for some absolute constant $\beta>0$. Then Theorem \ref{MainThm1} generalises to show that, in the small particle limit, the clusters arising from such a sequence converge to growing disks. The results in the next section can be generalised in a similar way.
\end{remark}

\section{Scaling limits and phase transition for harmonic measure flows}
\label{flowcon}

In the previous section we showed that, in the small particle limit, the macroscopic shape of the cluster converges to a disk. However, as clusters are formed by repeated aggregation of particles they have a complicated and (as we will see) interesting internal structure. In particular, there is a natural notion of ancestry for the particles and by tracing along ancestral lines we are able to identify random tree-like structures or branches within the cluster that correspond to the disconnected components of $K_n \cap \Delta$. In \cite{NT12} it is shown that these are closely related to the evolution of harmonic measure on the cluster boundary. Therefore, by analyzing the evolution of harmonic measure on the boundary for the $\mathrm{HL}(\alpha, \parsig)$ clusters as $\capc \to 0$ when $\parsig \gg (\log \capc^{-1})^{-1/2}$, we can in a certain sense gain insight into the internal branching structure of the cluster. 

In this section, we show that the harmonic measure flow exhibits a phase transition at $\alpha=0$. If $\alpha=0$, in \cite{NT12} it is shown that under appropriate scaling the harmonic measure flow converges to the Brownian web on the circle. As all Brownian motions on the circle starting at a fixed time eventually coalesce into a single Brownian motion, this intuitively corresponds to the $\mathrm{HL}(0)$ cluster having a single infinite branch, or equivalently all particles arriving beyond a certain time sharing a common ancestor. If $\alpha>0$ is fixed as $\capc \to 0$ we show that the harmonic measure flow of the $\mathrm{HL}(\alpha, \parsig)$ cluster converges to the identity flow which in this sense corresponds to the number of infinite branches becoming unbounded in the limit as $\capc \to 0$. We therefore consider different ``off-critical'' limits by letting $\alpha \to 0$ at different rates. We show that, on timescales of order $\capc^{-3/2}$, if $\alpha \ll \capc^{1/2}$, the harmonic measure flow 
converges to the Brownian web and the 
situation is the same as for $\alpha=0$; if $\alpha \gg \capc^{1/2}$, the harmonic measure flow converges to the identity flow and the situation is the same as for $\alpha>0$ fixed; whereas if $\alpha \capc^{-1/2} \to a \in (0,\infty)$, the harmonic measure flow converges to a time-change of the Brownian web, stopped at some finite time, which intuitively corresponds to a finite, but random number of infinite branches. By using a result of Bertoin and Le Gall \cite{BLG} that relates the Brownian web to Kingman's coalescent we are able to give the distribution of this number and show that it is stochastically increasing in $a$.  

For $n \geq m$ and $x \in \mathbb{R}$ define the harmonic measure flow $\bar{\Gamma}_{n,m}(x) = \IIm \tG_{n,m}(ix)$ where $\tG_{n,m}(z)$ is defined in Section \ref{prelim}. This map expresses how the harmonic measure on $\partial K_m$ is transformed by the arrival of new particles up to time $n$. Suppose that $0 \leq x<y<2\pi$ so that $\Phi_m(e^{ix}), \Phi_m(e^{iy}) \in \partial K_m$. Then the harmonic measure (from $\infty$) of the positively oriented boundary segment between these two points is given by $(y-x)/2\pi$ and after $n$ arrivals, the harmonic measure on $\partial K_n$ between these two points is given by $(\bar{\Gamma}_{n,m}(y)-\bar{\Gamma}_{n,m}(x))/2\pi$.

We define a timescale $N_{t} = \lfloor t\capc^{-3/2} \rfloor$, and 
stopping times
\[N(t) = \inf\{n : |c_{n+1}-c^*_{n+1}| > 2 \alpha \capc^{1+\beta}\} \wedge N_t  \]
where, provided that
\begin{equation}
\label{parsigcond} 
\parsig > (64 A_3 \log(1+ \alpha \capc^{-1/2} t))^{1/2} (\log \capc^{-1})^{-1/2},
\end{equation}
$\beta>0$ is the absolute constant guaranteed by Theorem \ref{capincrconv} such that $N(t) - N_{t} \to 0$ in probability
as $\capc \to 0$. 

In what follows, we assume that $\parsig \gg (\log \capc^{-1})^{-1/2}$ and that $\alpha \to 0$ sufficiently fast that for any fixed $t>0$, \eqref{parsigcond} eventually holds. Note that the arguments also hold when $\alpha \not\to 0$ but in this case we either need to bound $\parsig$ away from 0, or run the processes over shorter timescales.

\begin{theorem}
\label{ptwiseflowconvergence}
Suppose that $(s,x) \in [0, \infty) \times \RR$ are fixed. Then for $\capc \in (0, 1/2]$, the process $\left ( \bar{\Gamma}_{N(t), N(s)}(x) \right )_{s \leq t}$ is tight, and any weak limit as $\capc \to 0$ is a continuous local martingale starting from $x$ at time $s$.

Furthermore, provided $\parsig \gg (\log \capc^{-1})^{-1/2}$,
\begin{itemize}
\item If $\alpha \capc^{-1/2} \to \infty$ (sufficiently slowly that for any fixed $t>0$, \eqref{parsigcond} eventually holds), \[\left ( \bar{\Gamma}_{N_{t}, N_{s}}(x) \right )_{s \leq t} \to x\] uniformly on compacts in probability as $\capc \to 0$;
 
\item if $\alpha \capc^{-1/2} \to 0$, $\left ( \bar{\Gamma}_{N_{t}, N_{s} }(x) \right )_{s \leq t}$ converges in distribution to Brownian motion starting from $x$ at time $s$ with diffusivity $\frac{16}{3 \pi}$;

\item if $\alpha \capc^{-1/2} \to a$ for some $a \in (0,\infty)$, $\left ( \bar{\Gamma}_{N_{t}, N_{s}}(x) \right )_{s \leq t}$ converges in distribution to a time change of a standard Brownian motion stopped at $32/(3 \pi a)$ with time change given by
\[t \mapsto \frac{32}{3 \pi a}\left(1-\frac{1}{\sqrt{1+at}} \right).\]
\end{itemize}
\end{theorem}
\begin{proof}
First note that
\begin{align*}
 \bar{\Gamma}_{n,m}(x) &= x + \sum_{k=m+1}^n \tg_{c_k}(\bar{\Gamma}_{k-1,m}(x) - \theta_k) \\
                       &= x + \sum_{k=m+1}^n \tg_{c_k^*}(\bar{\Gamma}_{k-1,m}(x) - \theta_k) + \sum_{k=m+1}^n (\tg_{c_k}-\tg_{c_k^*})(\bar{\Gamma}_{k-1,m}(x) - \theta_k),
\end{align*}
and hence, for any fixed $T>0$, if $m \leq n \leq N(T)$ then by Lemma \ref{gmapestimate},
\begin{align*}
 &\left|\bar{\Gamma}_{n,m}(x)-x - \sum_{k=m+1}^n \tg_{c_k^*}(\bar{\Gamma}_{k-1,m}(x) - \theta_k)\right| \\
 &\leq 2 \sum_{k=m+1}^n \frac{|c_k - c_k^*|}{|\tan((\bar{\Gamma}_{k-1,m}(x) - \theta_k)/2)| \vee \sqrt{c_k \wedge c_k^*}} \\
 &\leq 4 \alpha \capc^{\beta} \sum_{k=m+1}^n \frac{c_k^*}{|\tan((\bar{\Gamma}_{k-1,m}(x) - \theta_k)/2)| \vee \sqrt{c_k^*/2}}.
 \end{align*}
Hence there exists an absolute constant $A$ such that
\begin{align*}
 &\EE \left ( \sup_{m \leq n \leq N(T)}|\bar{\Gamma}_{n,m}(x)-x - \sum_{k=m+1}^n \tg_{c_k^*}(\bar{\Gamma}_{k-1,m}(x) - \theta_k)| \right )\\
 &\leq 4 \alpha \capc^{\beta} \sum_{k=m+1}^{\lfloor T \capc^{-3/2} \rfloor}  \EE \left ( \frac{c_k^*}{|\tan((\bar{\Gamma}_{k-1,m}(x) - \theta_k)/2)| \vee \sqrt{c_k^*/2}} \right ) \\
 &= 4 \alpha \capc^{\beta} \sum_{k=m+1}^{\lfloor T \capc^{-3/2} \rfloor} \int_0^{2 \pi} \frac{c_k^*}{2 \pi |\tan(\theta/2)| \vee {\sqrt{c_k^*/2}}} d\theta \\
 &= \frac{4 \alpha \capc^{\beta}}{\pi} \sum_{k=m+1}^{\lfloor T \capc^{-3/2}\rfloor} c_k^* \log ((c_k^*)^{-1}) \left ( 1+\frac{2^{3/2} (c_k^*)^{-1/2}\tan^{-1}\sqrt{c_k^*/2} + \log(2+c_k^*)}{\log((c_k^*)^{-1})}\right) \\
 &\leq A \alpha \capc^{\beta} \int_0^{T \capc^{-3/2}} \frac{\capc}{1+\alpha \capc x} \log \frac{1 + \alpha \capc x}{\capc} dx \\
 &\leq A \capc^{\beta} \left ( \log \frac{1+\alpha T \capc^{-1/2}}{\capc} \right )^2 \\
 &\to 0
\end{align*}
as $\capc \to 0$.

So it remains to show that 
\[
Y_t = \sum_{k=N_{s}}^{N_{t}} \tg_{c_k^*}(\bar{\Gamma}_{k-1,N_{s}}(x) - \theta_k), \quad t \geq s,
\]
converges in distribution to a continuous martingale starting from $0$ at time $s$. Since 
\[
 \int_0^{2 \pi} \tg_{\capc}(\theta) d\theta = 0
\]
for all constants $\capc>0$, and $c_k^*$ is deterministic (independent of $\mathcal{F}_{k}$), it is straightforward to check that $Y_t$ is a martingale. 

If $s \leq r \leq t$, then recalling the definition of $\rho$ from \eqref{rhodef},
\begin{align}
\label{quadvar}
 \EE [(Y_t - Y_r)^2] &= \EE \left [ \left ( \sum_{k=N_r+1}^{N_t} \tg_{c_k^*}(\bar{\Gamma}_{k-1,N_s}(x) - \theta_k) \right)^2 \right ] \nonumber \\
 &= \sum_{N_r+1\leq j, k \leq N_t} \EE (\tg_{c_j^*}(\bar{\Gamma}_{j-1,N_s}(x) - \theta_j)\tg_{c_k^*}(\bar{\Gamma}_{k-1,N_s}(x) - \theta_k)) \nonumber \\
 &= \sum_{k=N_r+1}^{N_t} \EE (\tg_{c_k^*}(\bar{\Gamma}_{k-1,N_s}(x) - \theta_k)^2) \nonumber \\
 &= \sum_{k=N_r+1}^{N_t} \rho(c_k^*)^{-1}.
 \end{align}
 Note that for the second equality to hold it is again crucial that $c_k^*$ is independent of $\mathcal{F}_{k}$. 
 Therefore
 \begin{align*}
  \EE [(Y_t - Y_r)^2]&\leq A_2 \sum_{k=N_r+1}^{N_t} (c_k^*)^{3/2} \\
 &\leq A_2 (t-r), 
\end{align*}
 where $A_2$ is the absolute constant from Lemma \ref{gmapestimatesnt}.
  Hence, by Aldous' criterion (see, for example, \cite[Theorem 16.11]{KallBook}), the family of processes $\{ (Y_t)_{t \geq s} : \capc \in (0, 1/2)\}$ is tight. Also, by Lemma \ref{gmapestimatesnt}, the jumps in $Y_t$ are bounded in absolute value by $A_2 \capc^{1/2}$. Let $\mu$ be any weak limit law for the limit $\capc\to 0$.
Write $(Z_t)_{t\ge s}$ for the coordinate process on $\{f \in D([s,\infty),\RR): f(s)=0\}$.
Then
%, by standard arguments, 
$\mu$ is supported on continuous paths and under $\mu$,
$(Z_t)_{t\ge s}$ is a local martingale in the natural filtration of $(Z_t)_{t\ge s}$. 

By \eqref{slitrho} 
$$
 \left | \sum_{k=N_r+1}^{N_t} \rho(c_k^*)^{-1} - \frac{16}{3 \pi} \sum_{k=N_r+1}^{N_t} (c_k^*)^{3/2}  \right | 
 \leq (t-r) \sup_{k \in \NN} \left|(c_k^*)^{-3/2}\rho(c_k^*)^{-1} - \frac{16}{3\pi}\right| \\
 \to 0
$$
 and
 
 \begin{align*}
   \left | \sum_{k=N_r+1}^{N_t} (c_k^*)^{3/2}  \right. &- \left.  \frac{2 \capc^{1/2}}{\alpha} \left ( \frac{1}{\sqrt{1+\alpha \capc^{-1/2} r}} - \frac{1}{\sqrt{1+ \alpha \capc^{-1/2} t}}\right ) \right | \\
  & = \left | \sum_{k=N_r+1}^{N_t} (c_k^*)^{3/2}  -  \int_{r \capc^{-3/2}}^{t \capc^{-3/2}} \frac{\capc^{3/2}}{(1+\alpha \capc x)^{3/2}} dx \right | \\
  &\leq  \frac{3}{2} (t-r) \alpha \capc \\
  &\to 0.
 \end{align*}
Recalling \eqref{quadvar}, we obtain that
\begin{eqnarray*}
\EE [(Y_t - Y_r)^2] & \to \begin{cases}
                              \frac{16}{3 \pi } (t-r), &\quad \text{if } \alpha \capc^{-1/2} \to 0; \\
                              \frac{32}{3 \pi a} \left ( \frac{1}{\sqrt{1+ar}} - \frac{1}{\sqrt{1+at}}\right ), &\quad \text{if } \alpha \capc^{-1/2} \to a;\\%\in (0, \infty); \\
                              0, &\quad \text{if } \alpha \capc^{-1/2} \to \infty,
                            \end{cases}
\end{eqnarray*}
as $\capc \to 0$.

Hence, if $\alpha \capc^{-1/2} \to \infty$ then, provided $N(t) \to N_t$ uniformly on compacts in probability (which is ensured by the assumptions and Theorem \ref{capincrconv}), we have $\left ( \bar{\Gamma}_{N_t, N_s}(x) \right )_{s \leq t} \to x$ in probability. % \bg Fluctuations in the case of $\alpha$ constant come from errors in the approximation $|c_k - c_k^*|$ so are of order $O(\capc^{\beta} \log (\capc^{-1}))$ and would be worth looking at in greater detail at some stage.\eg

Similarly, if $\alpha \capc^{-1/2} \to 0$, then for any weak limit measure $\mu$, the coordinate function $Z_t$ has the property that $(Z_t^2 - \frac{16}{3 \pi}t)_{t \geq s}$ is a martingale and so, by L\'evy's characterization of Brownian motion,$\left ( \bar{\Gamma}_{N_t, N_s }(x) \right )_{s \leq t}$ 
converges in distribution to Brownian motion starting from $x$ at time $s$ with diffusivity $16/(3 \pi)$.

If $\alpha \capc^{-1/2} \to a$ for some $a \in (0, \infty)$, then $\left(Z_t^2 - \frac{32}{3 \pi a}\left (1-\frac{1}{\sqrt{1+at}}\right)\right)_{t \geq s}$ is a martingale and so $\left ( \bar{\Gamma}_{N_t, N_s }(x) \right )_{s \leq t}$ converges in distribution to a time change of a standard Brownian motion $W(t)$ stopped at $32/(3 \pi a)$, that is,
$$
\left ( \bar{\Gamma}_{N_t, N_s}(x) \right )_{s \leq t} 
\Rightarrow \left (x + W\left (\frac{32}{3 \pi a}\left(1-\frac{1}{\sqrt{1+at}}\right)\right)- W\left(\frac{32}{3 \pi a}\left(1-\frac{1}{\sqrt{1+as}}\right)\right)\right)_{t \geq s}.
$$
\end{proof}

The following theorem extends the result above to show that if $\alpha \capc^{-1/2} \to a \in [0, \infty)$, the harmonic measure flow converges to a time change of the coalescing Brownian flow (more commonly known as the Brownian web). Due to the discontinuities in time of the harmonic measure flow, it is convenient to work in the flow space $D^\circ([0,\infty),\mathcal{D})$ that is defined in \cite{NTprep}. The construction of this space and criteria for proving convergence can be found in \cite{NTprep}. 

\begin{theorem}\label{flowconvergence}
Suppose that $\parsig \gg (\log \capc^{-1})^{-1/2}$. If $\alpha \capc^{-1/2} \to 0$ then the rescaled harmonic measure flow \[\left ( \bar{\Gamma}_{N_t, N_s }(x) \right )_{s \leq t}\] converges weakly in the flow space $D^\circ([0,\infty),\mathcal{D})$ to the coalescing Brownian flow on the circle with diffusivity $\frac{16}{3 \pi}$.

If $\alpha \capc^{-1/2} \to a$ for some $a \in (0, \infty)$ then the rescaled harmonic measure flow $\left ( \bar{\Gamma}_{N_t, N_s }(x) \right )_{s \leq t}$ converges weakly in the flow space $D^\circ([0,\infty),\mathcal{D})$ to a time change of the coalescing Brownian flow on the circle stopped at 
$32/(3 \pi a)$ with time change given by $t \mapsto \frac{32}{3 \pi a}\left(1-\frac{1}{\sqrt{1+at}}\right)$.
\end{theorem}
\begin{proof}
This proof is an adaptation of \cite[Proposition 2.3]{NTprep}.

 For $e=(s,x) \in [0,\infty) \times \RR$, let 
\[(X_t^e)_{t \geq s}= \left ( \bar{\Gamma}_{N_t, N_s }(x) \right )_{s \leq t}.\] If $E = \{e_1, e_2, \dots \}$ is a countable sequence in $[0,\infty) \times \RR$, then by a similar argument to above, the family of processes $\left( (X_t^{e_i})_{t \geq s} \right )_{i \in \NN}$ is tight. If $\mu$ is any limit measure then it is straightforward using the estimate 
 \[
 \frac{1}{2\pi} \int_0^{2 \pi} |\tg_\capc(x) \tg_\capc(x+h)| dx \leq \frac{A_2 \capc^2}{h} \log \left ( \frac{1}{\capc}\right )
 \]
 whenever $h \in [d, \pi]$ from Lemma \ref{gmapestimatesnt} to show that under $\mu$, for any $i,j$ the product of coordinate processes $(Z_t^iZ_t^j)_{s_i \vee s_j \leq t \leq T^{ij}}$ is a local martingale where $T^{ij}$ is the time of first coalescence of the processes $Z_t^i, Z_t^j$ on the circle. We also note that $\mu$ 
inherits from the laws of $\left( (X_t^{e_i})_{t \geq s} \right )_{i \in \NN}$ the 
property that, with probability $1$, for all $n\in\ZZ$, 
the process $(Z^i_t-Z^j_t+2 \pi n:t\ge s_i\vee s_j)$ has no change in sign.
So, by an optional stopping argument, $Z^i_t-Z^j_t$ is constant for $t\ge T^{ij}$. This means that the finite dimensional distributions of the rescaled harmonic measure flow converge in distribution to those of the (time changed) coalescing Brownian flow and hence, by Theorem 6.1 in \cite{NTprep}, the rescaled harmonic measure flow has the claimed limit distribution. 
\end{proof}

Suppose that $0\leq x < y< 2 \pi$. Then, after the arrival of $N_t$ particles, the harmonic measure of the part of the cluster that has grown between $e^{ix}$ and $e^{iy}$ is given by $(\bar{\Gamma}_{N_t, 0 }(x) - \bar{\Gamma}_{N_t, 0 }(y))/2\pi$. Therefore, if the paths $(\bar{\Gamma}_{N_t, 0 }(x))_{t>0}$ and $(\bar{\Gamma}_{N_t, 0 }(y))_{t>0}$ coalesce, then no further particles will attach to the cluster between them and hence only finite branches are rooted between $e^{ix}$ and $e^{iy}$. This suggests a correpondence between the number of infinite branches in a cluster and the number of gaps, or discontinuities in the flow, that survive infinitely long. We call $x \in [0, 2 \pi)$ is a ``common ancestor'' point if $\bar{\Gamma}_{N_t, 0 }(x)$ is discontinuous at $x$ for all $t>0$. The following result from \cite{BLG} gives the distribution of the number of discontinuities in the coalescing Brownian flow at each time, which can be applied to the above result to obtain the distribution of the number of 
common ancestors in the limit $\mathrm{HL}(\alpha,\parsig)$ cluster.

\begin{Mth}[Bertoin and Le Gall, 2005]
 Suppose that $X^1_0, \dots, X^p_0$ are uniformly distributed on the unit circle (identified with $[0,1)$) and suppose that $X^1_t, \dots, X^p_t$ are coalescing Brownian motions with diffusion coefficient $\sqrt{1/12}$ starting from $X^1_0, \dots, X^p_0$. Define a partition $\Pi^p_t$ on $\{1, 2, \dots, p \}$ by $i \sim j$ if and only if $X_t^i=X_t^j$. Then the process $\Pi^p_t$ is Kingman's coalescent.
\end{Mth}

Well known properties of Kingman's coalescent (see \cite{Nat09} for an overview of Kingman's 
coalescent and related processes) therefore imply the following.  

\begin{Pro}
Suppose that $\parsig \gg (\log \capc^{-1})^{-1/2}$ and $\alpha \capc^{-1/2} \to a$ for some $a \in [0, \infty)$. Let $B$ be the number of common ancestors in the limit $\mathrm{HL}(\alpha,\parsig)$ cluster as $\capc \to 0$, that is the number of points $x$ at which $\lim_{\capc \to 0}\bar{\Gamma}_{N_t, 0 }(x)$ is discontinuous for all $t>0$.

If $a=0$, then $B=1$ a.s.

Suppose $a>0$ and let $\tau_j$ be the time of coalescence into $j$ partitions in Kingman's coalescent. 
Then,
\[
 \mathbb{P}(B\leq j) = \mathbb{P}(\tau_j \leq 8/(9\pi a)),
\]
and, in particular, the distribution of $B$ is stochastically increasing in $a$.

Furthermore, the distribution of $\tau_j$ can be explicitly calculated by
 \[
  \tau_j \sim \sum_{k=j+1}^{\infty} E_k
 \]
where $E_k$ are independent exponential random variables with rates $k(k-1)/2$. Conditional on $B=j$, the positions of the $j$ common ancestors are that of $j$ independent uniform points on $[0, 2\pi)$.
\end{Pro}

\end{document}